\documentclass{article}

\RequirePackage[OT1]{fontenc}
\RequirePackage{amsmath,amsxtra,amscd,amssymb,latexsym,stmaryrd,amsthm}
\RequirePackage[english,french]{babel}
\RequirePackage[colorlinks,citecolor=blue,urlcolor=blue]{hyperref}
\RequirePackage{hypernat}
\usepackage{graphicx,transparent}
\usepackage{enumerate}  
\usepackage{hyperref}
\usepackage{dsfont}
\usepackage{geometry}

\def \m{\mu} \def \ta{\tau} \def \n{\nu}
\def \a{\alpha}       \def \d{\delta}     
   \def \s{\sigma}   \def \t{\theta}  
\def \r{\rho}  
\def \e{\varepsilon}

\def \R{\mathbb R} \def \H{\mathbb H}  
   \def \S{\mathbb S}

  \def \MM{{\cal M}}

\def \pa{\partial}        \def \ii{\infty}

   \def \ds{\displaystyle}   \def \ts{\textstyle}       
\def \rt1{\sqrt{-1}\,\,}  \def \1{^{-1}}            \def \2{^{-2}}             \def \5{{\ts {1\over 2}}}

\def \ash{\mathrm{argsh}\,}

\numberwithin{equation}{section}
\theoremstyle{plain}
\newtheorem{thm}{Theorem}[section]
\newtheorem{lemma}{Lemma}[section]
\newtheorem{prop}{Proposition}[section]
\newtheorem{cor}{Corollary}[section]
\newtheorem{hypo}{Hypothesis}
\theoremstyle{remark}
\newtheorem{rem}{Remark}[section]
\theoremstyle{definition}
\newtheorem{dfn}{Definition}[section]
\newtheorem{exa}{Example}[section]

\begin{document}

\selectlanguage{english}
\renewcommand\proofname{Proof}



\title{Asymptotic behavior of a relativistic diffusion in Robertson-Walker space-times}

\author{J\"urgen Angst
\footnote{address: IRMAR, Universit\'e Rennes 1, bureau 320, Bat 22, Campus de Beaulieu, 35042 Rennes Cedex, France, email: \url{jurgen.angst@univ-rennes1.fr}}}

\maketitle

\begin{abstract}
We determine the long-time asymptotic behavior of a relativistic diffusion taking values in the unitary tangent bundle of a Robertson-Walker space-time. We prove in particular that when approaching the explosion time of the diffusion, its projection on the base manifold almost surely converges to a random point of the causal boundary and we also describe the behavior of the tangent vector in the neighborhood of this limiting point. 
\end{abstract}

\selectlanguage{english}





\section{Introduction}
The study of Brownian motion on a Riemannian manifold shows that the short-time and long-time asymptotic behavior of the process strongly reflects the geometry of the underlying manifold. Considering the importance of heat kernels in Riemannian geometry, it appears very natural to investigate the links between geometry and asymptotics of Brownian paths in a Lorentzian setting. 
In his seminal work \cite{dudley1}, R.M. Dudley showed that a relativistic diffusion, \emph{i.e.} a diffusion process with values in a Lorentz manifold whose law is Lorentz-covariant, cannot exist in the base space, but makes sense at the level of the tangent bundle. More precisely, Dudley showed that there is no Lorentz-covariant diffusion in the Minkowski space-time but that there exists a unique Lorentz-covariant diffusion with values in its (pseudo)-unitary tangent bundle.
This process, that we will name \textit{Dudley's diffusion} in the sequel, is simply obtained by integrating the classical hyperbolic Brownian motion on the unitary tangent space. 
\par
\medskip
In \cite{flj}, J. Franchi and Y. Le Jan extend Dudley's construction to the realm of general relativity by defining, on the 
future-directed half of the unitary tangent bundle $T^1_+ \mathcal M$ of an arbitrary Lorentz manifold $\mathcal M$, a diffusion which is Lorentz-covariant. This process, that we will simply call \textit{relativistic diffusion}, is the Lorentzian
analogue of the classical Brownian motion on a Riemannian manifold. It can be seen either as a random perturbation of the timelike geodesic flow on the unitary tangent bundle, or as a stochastic development of Dudley's diffusion in a fixed tangent space, see Sect. 3 of \cite{flj}. In \cite{flj2}, Franchi and Le Jan generalized their construction by introducing the so-called ``curvature diffusions'', whose quadratic variation is allowed to depend locally on the curvature of the underlying space-time.  
\par 
\medskip
In the case when the underlying manifold is the Minkowski space-time, the long-time asymptotic behavior of the above relativistic diffusion is well understood. It was first studied by Dudley himself in \cite{dudley1,dudley3} where it is shown that the process is transient, and escapes to infinity in a random preferred direction. In \cite{ismael}, I. Bailleul completed Dudley's results by performing the full determination of the Poisson
boundary of the relativistic diffusion, \emph{i.e.} the set of bounded harmonic functions on the Minkowski phase space endowed with the differential operator which is the infinitesimal generator of the diffusion. Recall that this is equivalent to the determination of the invariant $\sigma$-field of the natural filtration of the relativistic diffusion. Moreover, Bailleul gave a geometric description of the Poisson boundary of the relativistic diffusion which can be formulated in terms of the causal boundary of Minkowski space-time. 
\par
\medskip
As for the usual Brownian motion on a general Riemannian manifold, there is no hope to fully determine the asymptotic behavior of the relativistic diffusion on an arbitrary Lorentzian manifold: it could depend heavily on the base space, see e.g. \cite{marcanton} and its references in the case of Cartan--Hadamard manifolds. In fact, the difficulty is a priori greater in the Lorentzian context: first because of the non-positivity of the underlying metric, then because the relativistic diffusion does not live on the base manifold, but on its pseudo-unit tangent bundle, so that it is basically seven-dimensional when the base manifold have four dimensions, and there is no general reason for which it must contain one or more lower-dimensional sub-diffusions. On the contrary, recall that in the case of a constantly curved Riemannian manifold, the Brownian motion fortunately always admits a one-dimensional sub-diffusion: the radial sub-diffusion. 
\par
\medskip
Nevertheless, the study of the relativistic diffusion has been led in details in a few examples of Lorentzian manifolds.
Thereby, in \cite{flj} and \cite{franchigoedel}, the authors studied the long-time behavior of the diffusion in Schwarzschild-Kruskal-Szekeres
space-time and G\"odel space-time respectively. Although they did not reach the full determination of the
Poisson boundary, they achieved to describe the almost sure asymptotics of diffusion's paths and came up with the conclusion that they asymptotically behave like random light-like geodesics. 
\par
\medskip
The purpose of this paper is to perform a detailed study of the long-time asymptotic behavior of the relativistic diffusion in the case when the underlying space-time belong to a large class of Lorentz manifolds : Robertson-Walker space-times, or RW space-times for short, see Sect. \ref{sec.RW}. This class of Lorentzian manifolds offers the advantage of being very rich (e.g., RW space-times can be spatially compact / non compact, geometrically complete / non complete etc.), yet the geometry of these space-times remains quite simple (one can for example explicitly integrate the geodesic equations).
\par
\medskip
We fully caracterize the asympotic behavior of the diffusion in terms of geometric properties at infinity of the base manifold. 
We show in particular that the relativistic diffusion's paths converge almost surely to random points of the causal boundary $\partial \mathcal M_c^+$ (see Sect. \ref{sec.RW} and \cite{geroch,flores}) of the base manifold $\mathcal M$.

\newtheorem{thmintro}{Theorem}
\renewcommand{\thethmintro}{\empty{}}
\begin{thmintro}[Theorem 3.1 below]
Let $\mathcal M:=(0, T) \times_{\alpha} M$ be a RW space-time satisfying the hypotheses of Sect. \ref{sec.RW}. Let $(\xi_0, \dot{\xi}_0) \in T^1_+ \mathcal M$ and let $(\xi_s, \dot{\xi}_s)_{0 \leq s \leq \tau}$ be the relativistic diffusion in $T^1_+\mathcal M$ starting from $(\xi_0, \dot{\xi}_0)$. Then, almost surely as $s$ goes to the explosion time $\tau$ of the diffusion, the process $\xi_s$ converges to a random point $\xi_{\infty}$ of the causal boundary $\partial \mathcal M_c^+$.
\end{thmintro}
\par
\smallskip
As the geometry of the causal boundary strongly reflects the one of the base manifold, the above synthetic result actually covers a huge variety of geometric asymptotic behaviors, depending on the type of RW space-times considered, see Sect. \ref{sec.conv}. 
\par
\smallskip
We also caracterize the asymptotic behavior of the tangent vector $\dot{\xi}_s \in T_{\xi_s}^1 \mathcal M$ when $s$ goes to $\ta$. 
We show in particular that it is strongly related to the finiteness of $T$ and/or to the growth rate of the torsion function $\a$.
Rougthly speaking, if $T<+\infty$ or $T=+\infty$ and $\a$ grows polynomially, then when properly rescaled, the tangent process $\dot{\xi}_s$ is convergent, whereas if $T=+\infty$ and $\a$ grows exponentially fast, $\dot{\xi}_s$ shows some ``recurrence'' properties.
Precise statements are given in section \ref{sec.derivative} below.  
\par
\smallskip
\begin{rem}
In the case of a spatially flat Robertson-Walker space-time $\mathcal M= (0, +\infty) \times_{\alpha} \mathbb R^{d}$ where $\a$ has exponential growth, we manage to push the analysis further, see \cite{angst1}. Indeed, we show that not only the process $\xi_s$ converges to a random point $\xi_{\infty}$ of the causal boundary $\partial \mathcal M_c^+$ but the Poisson boundary of the diffusion is precisely generated by the single random variable $\xi_{\infty}$, which in that case can be identified with a spacelike copy of the Euclidean space $\mathbb R^{d}$. Namely, we prove the following result:
\end{rem}
\par
\begin{thmintro}[Theorems 3.3 and 3.4 of \cite{angst1}]
Let $\mathcal M:=(0, +\infty) \times_{\alpha} \mathbb R^{d}$ be a Robertson-Walker space-time where $\alpha$ has exponential growth. Let $(\xi_0, \dot{\xi}_0) \in T^1_+ \mathcal M$ and let  $(\xi_s, \dot{\xi}_s)_{s \geq 0}=(t_s,x_s, \dot{t}_s, \dot{x}_s)_{s \geq 0}$ be the relativistic diffusion in $T^1\mathcal M$ starting from $(\xi_0, \dot{\xi}_0)$. Then, almost surely when $s$ goes to infinity, the spatial projection $x_s$ converges to a random point $x_{\infty}$ in $\mathbb R^{d}$, and the invariant $\sigma-$field of the whole diffusion is generated by the single random variable $x_{\infty}$.
\end{thmintro}
\par
The article is organized as follows. In the next section, we briefly recall the geometrical background on Robertson-Walker space-times and the definition of the relativistic diffusion in this setting. In Section \ref{sec.results}, we then state the results concerning the asymptotic behavior of the relativistic diffusion. The last section is devoted to the proofs of these results. For reasons of concisions, some elements of proofs are omitted here but they appear in great detail in the author's thesis \cite{mathese}.


\section{Geometrical and probabilistic background}\label{sec.background}
The Lorentz manifolds we consider here are Robertson-Walker space-times. These manifolds are named after H. P. Robertson and A. G. Walker \cite{robertson,walker} and their work on solutions of Einstein's equations satisfying the ``cosmological principle''. They are the natural geometric framework to formulate the theory of Big-Bang in General Relativity.

\subsection{Robertson-Walker spacetimes}\label{sec.RW}
The constraint that a space-time satisfies both Einstein's equations and the cosmological principle implies it has a warped product structure, see e.g. \cite{Weinberg} p. 395--404. A RW space-time, classically denoted by $\mathcal M:=I \times_{\alpha} M$, is thus defined as a Cartesian product of an open interval $(I,-dt^2)$ (the base) and a Riemannian manifold $(M, h)$ of constant curvature (the fiber), endowed with a Lorentz metric of the following form $g := -dt^2 + \alpha^2(t)h,$ where $\alpha$ is a positive function on $I$, called the \textit{expansion function} or \textit{torsion function}. A general study on the geometry of warped product manifolds can be found in \cite{ghani1}. More specific results on the geometry of RW space-times can be found in \cite{floressanchez}. Classical examples of RW space-times are the (half)$-$Minkowski space-time, Einstein static universe, de Sitter and anti-de Sitter space-times etc.
\par
\smallskip
Let us fix an integer $d \geq 3$. A smooth, simply connected $d-$dimensional Riemannian manifold $(M, h)$ of constant curvature is either isometric to the Euclidean space $\mathbb R^{d}$, the hyperbolic space $\mathbb H^{d}$, or the Euclidean sphere $\mathbb S^{d}$, with their standard metric structures. Without loss of generality, we will thus restrict ourself to these three cases, and in the sequel $(M,h)$ will denote one of these three spaces endowed with their standard associated metric.  
Moreover, we will consider here two types of expansion functions that correspond to the two standard models in cosmology: a universe in \emph{infinite expansion} or a \emph{Big-Crunch}, see Fig. \ref{fig.alpha}. More precisely, we will assume the following mild hypotheses on $\alpha$:
\begin{hypo}\label{hypo.1}
\noindent
\begin{enumerate}
\item The function $\alpha$ is positive and of class $C^2$ on $(0, T)$;
\item The function $\alpha$ is $\mathrm{log}-$concave, \emph{i.e.} the Hubble function $H:=\alpha'/\alpha$ is nonincreasing;
\item We are in one of these two cases :
\begin{enumerate}
\item $T=+\infty$, and $H \geq 0$ on $(0, +\infty)$ (infinite expansion);
\item $T<+\infty$ and $\lim_{t \to 0^+} \a(t)= \lim_{t \to T^-}\a(t)=0$ and $\lim_{t \to T^-} H(t)=-\infty$ (Big-Crunch);
\end{enumerate}

\end{enumerate}
\end{hypo}

\begin{figure}[ht]
\hspace{3.5cm}\includegraphics[scale=0.5]{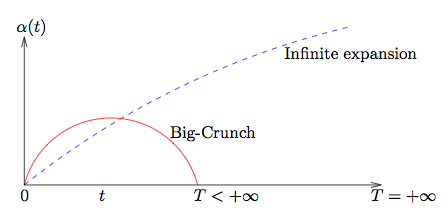}
\caption{The two types of expansion functions considered.\label{fig.alpha}}
\end{figure}
In the case $T=+\infty$, since $H$ is nonnegative and nonincreasing, it admits a limit at infinity that we denote by $H_{\infty}:=\lim_{t \to +\infty} H(t)$.

\begin{rem}\label{rem.logconcave}
The hypothesis of $\mathrm{log}-$concavity of the expansion function is classical. From a physical point of view, it ensures that a RW space-time satisfies the weak energy condition of \cite{hawell}. Indeed, the stress-energy tensor associated to a RW space-time via Einstein's Equations has a perfect fluid structure with energy density $\mathfrak q$ and pressure density $\mathfrak p$ given by
\[
8 \pi \mathfrak q /3 :=  \left(\frac{\a'^2(t)}{\a^2(t)}+\frac{k}{\a^2(t)}\right), \quad -8 \pi \mathfrak p := \left( \frac{2\a''(t)}{\a(t)} + \frac{\a'^2(t)}{\a^2(t)}+\frac{k}{\a^2(t)}\right),
\]
where $k \in \{-1, 0, 1\}$ is the curvature of the fiber $M$. The classical weak energy condition $\mathfrak q+\mathfrak p \geq 0$ and strong energy condition $\mathfrak q+3 \mathfrak p \geq 0$ are thus respectively equivalent to 
\[
-2 \left( \frac{\a''(t)}{\a(t)} - \frac{\a'^2(t)}{\a^2(t)}+\frac{2k}{\a^2(t)}\right) \geq 0, \;\; \hbox{and} \;\; - \alpha''(t) \geq 0.
\]
In particular, a spatially flat RW space-time satisfies the weak energy condition if and only if the expansion function is log-concave. If the warping function is concave, the strong energy condition is automatically satisfied, whatever the curvature of the fiber.
From a more mathematical point of view, $\mathrm{log}-$concavity of the expansion function allows for example to caracterize spacelike hypersurfaces with constant higher order mean curvature as the slices $\{ t_0 \}\times M $ of the foliation $I \times_{\alpha} M$ \cite{alias}. 
Let us also note that despite we are assuming global $\mathrm{log}-$concavity here for simplicity, most of our results generalize if $\alpha$ is $\mathrm{log}-$concave outside a compact set.\end{rem}

In the case $T=+\infty$, the growth rate at infinity of the warping function $\alpha$ will play and important role in the sequel.

\begin{dfn}\label{defi.2}
If $T=+\infty$, we will say that the growth rate of the torsion function is 
\begin{enumerate}
\item (at most) polynomial if
\[ 
H_{\infty}=0 \quad \hbox{and} \quad \limsup_{t \to +\infty} \frac{\log \left( \alpha(t)\right)}{\log \left(\int^t \alpha(u) du\right)}<1;
\]
\item subexponential if 
\[ 
H_{\infty}=0 \quad \hbox{and} \quad \lim_{t \to +\infty} \frac{\log \left( \alpha(t)\right)}{\log \left(\int^t \alpha(u) du\right)}=1;
\]
\item exponential if $H_{\infty}>0$.
\end{enumerate}
\end{dfn}

Recalling that $\alpha$ is $\mathrm{log}-$concave so that $H'\leq 0$, a simple integration by parts shows that 
\begin{equation}\label{eqn.loglog}
H(t) \leq \frac{\alpha(t)}{\int^t \left( 1 - \frac{H'(u)}{H(u)^2} \right) \alpha(u) du}  \;\; \hbox{and thus by integration}\;\; \frac{\log \left( \alpha(t)\right)}{\log \left(\int^t \alpha(u) du\right)} \leq 1.
\end{equation}
Thus, in the case $T=+\infty$ and $H_{\infty}=0$, the distinction between an at most polynomial growth and a subexponential growth is related to the asymptotic behavior of the ratio $H'(t)/H(t)^2$. For example, if $\alpha(t)=t^c$ with $0<c<+\infty$ i.e. $H(t)=  t / c$ and $ -H'(t)/H(t)^2 \equiv 1/c$, when $t$ goes to infinity we have 
\[
\frac{\log \left( \alpha(t)\right)}{\log \left(\int^t \alpha(u) du\right)} \longrightarrow \frac{c}{1+c},
\]
and the expansion is of course at most polynomial in the sense of Definition \ref{defi.2}. Therefore, if the growth of $\alpha$ is subexponential, we have necesserally $\liminf_{t \to +\infty} - H'(t)/H(t)^2=0$ and in most representative examples, we have in fact  $\lim_{t \to +\infty} H'(t)/H(t)^2=0$, see the first case in Example \ref{exa.hypo2} below. Nevertheless, an oscillatory behavior is not forbidden, see the second example below. To avoid very pathological examples in the subexponential case, we will make the following extra assumption.

\begin{hypo}\label{hypo.2}
If $T=+\infty$ and the growth of $\alpha$ is subexponential, then
\[
\liminf_{t \to +\infty} - \frac{H'(t)}{H(t)^2}  =0 \quad \hbox{and} \quad  \limsup_{t \to +\infty} - \frac{H'(t)}{H(t)^2}  =\kappa \in [0,+\infty).
\]
\end{hypo}

\begin{exa}\label{exa.hypo2}
To illustrate the fact that this extra assumption is not restrictive, let us give examples of warping functions satisfying both Hypotheses \ref{hypo.1} and \ref{hypo.2}.
\begin{enumerate}
\item  if $\alpha(t)=t^{\gamma} \exp \left( t^\beta \right)$ with $\gamma \in [0,+\infty)$ and $\beta \in (0,1)$, then
\[
H(t)=\frac{\gamma}{ t}+\frac{\beta}{ t^{1-\beta}} \to 0 \;\; \hbox{and}\;\; - \frac{H'(t)}{H^2(t)} = \frac{\gamma + \beta(1-\beta)t^{\beta}}{(\gamma + \beta t^{\beta})^2}\to 0 \;\; \hbox{so that}\;\;\lim_{t \to +\infty} \frac{\log \left( \alpha(t)\right)}{\log \left(\int^t \alpha(u) du\right)}=1;
\]
\item consider the triangular function $\phi(t):=(1+t)\mathds{1}_{t \in [-1,0]} + (1-t)\mathds{1}_{t \in [0,1]}$ and define the sawtooth function
\[ 
- \frac{H'(t)}{H^2(t)} := \kappa \sum_{k\in \mathbb N \backslash \{0\}} \phi(k(t-k)), \quad \hbox{so that}\quad H(t) \simeq \frac{\kappa}{\log(t)} \;\; \hbox{at infinity}.
\]
We have then naturally 
\[
0=\liminf_{t \to +\infty} - \frac{H'(t)}{H(t)^2}  <  \limsup_{t \to +\infty} - \frac{H'(t)}{H(t)^2}  =\kappa, \quad  \hbox{and}\;\;\lim_{t \to +\infty} \frac{\log \left( \alpha(t)\right)}{\log \left(\int^t \alpha(u) du\right)}=1.
\]
\end{enumerate}
\end{exa}

\par
\smallskip
The geodesic completeness and the geometry at infinity of a RW space-time $\mathcal M:=I \times_{\alpha} M$ are strongly related to the finiteness of the 
following integrals, which will play a major role in the description of the asymptotic behavior of the relativistic diffusion : 
\begin{equation}I_{-}(\alpha):=\int_0^{c} \frac{du}{\alpha(u)} \in \mathbb R^+ \cup \{+ \infty\}, \qquad I_{+}(\alpha):=\int_{c}^T \frac{du}{\alpha(u)} \in \mathbb R^+ \cup \{+ \infty\}, \;\; \hbox{where} \;\; c \in (0,T). \end{equation}
\par
Robertson-Walker space-times are classical examples of globally hyperbolic and thus strongly causal space-times. Such spaces admit a natural and intrinsic compactification called the causal boundary which was first introduced by Geroch, Kronheimer and Penrose in \cite{geroch}. In their approach, a future (past) ideal point is attached to every inextensible, physically admissible future (past) trajectory, in such a way that the ideal point only depends on the past (future) of the trajectory, see \cite{harris}. The resulting causal boundary $\pa \mathcal M_c =\pa \mathcal M_c^- \cup \pa \mathcal M_c^+$ decomposes into the union of two partial boundaries, $\pa \mathcal M_c^-$ corresponding to past oriented trajectories and $\pa \mathcal M_c^+$ corresponding to future oriented ones. In the case when the underlying space-time $\mathcal M$ is a RW space-time, the causal boundary $\pa \mathcal M_c$ was explicited in \cite{flores}. It depends on finiteness of 
$I_{-}(\alpha)$ and $I_{+}(\alpha)$ and on the curvature of the Riemannian fiber $M$. 

\begin{thmintro}[Theorems 4.2 and 4.3 of \cite{flores}] The causal boundary $\pa \mathcal M_c =\pa \mathcal M_c^- \cup \pa \mathcal M_c^+$ of a RW space-time $\mathcal M:=(0, T) \times_{\alpha} M$  has the following structure:
\begin{enumerate}
\item  \underline{Case $I_{-}(\alpha)=I_{+}(\alpha) = +\infty$}. If $M =\mathbb R^{d}$ or $\mathbb H^{d}$, then $\pa \mathcal M_c^-$ and $\pa \mathcal M_c^+$ are formed by two infinity null cones,  with base $\mathbb S^{d-1}$ and apex $i_-$ and $i_{+}$ respectively, where $i_+$ (resp. $i_{-}$) corresponds to the future (resp. past) timelike infinity. If $M =\mathbb S^{d}$, then $\pa \mathcal M_c^-$ and $\pa \mathcal M_c^+$ are just formed by $i_-$ and $i_{+}$ respectively.
\item\underline{Case $I_{-}(\alpha)<+\infty$, $I_{+}(\alpha) <+\infty$}. 
In that case, $\pa \mathcal M_c^-$ and $\pa \mathcal M_c^+$ are formed by two spacelike copies of $\mathbb R^{d}$ if $M =\mathbb R^{d}$ or $\mathbb H^{d}$,  or by two spacelike copies of $\mathbb S^{d}$ if $M = \mathbb S^{d}$.
\item \underline{Case $I_{-}(\alpha)<+\infty$, $I_{+}(\alpha) =+\infty$}.
If $M=\mathbb R^{d}$ or $\mathbb H^{d}$, then $\pa \mathcal M_c^+$ is formed by an infinity null
cone with base $\mathbb S^{d-1}$ and apex $i_+$,  and $\pa \mathcal M_c^-$ is formed by a spacelike copy $\mathbb R^{d}$.
If $M = \mathbb S^{d}$, then $\pa \mathcal M_c^+$ is just formed by $i_+$ and $\pa \mathcal M_c^-$ is formed by a spacelike copy of $\mathbb S^{d}$.
\item \underline{Case $I_{-}(\alpha)=+\infty$, $I_{+}(\alpha) <+\infty$}. 
If $M=\mathbb R^{d}$ or $\mathbb H^{d}$, then $\pa \mathcal M_c^-$ is formed by an infinity null
cone with base $\mathbb S^{d-1}$ and apex $i_-$,  and $\pa \mathcal M_c^+$ is formed by a spacelike copy $\mathbb R^{d}$.
If $M = \mathbb S^{d}$, then $\pa \mathcal M_c^-$ is just formed by $i_-$ and $\pa \mathcal M_c^+$ is formed by a spacelike copy of $\mathbb S^{d}$.
\end{enumerate}
\end{thmintro}
We give now representative examples of each of the above cases. These examples are standard models in cosmology, they are obtained by solving Friedmann Equations under a constraint of the form $\mathfrak p= c \, \mathfrak q$ for some constant $c$, where $\mathfrak q$ and $\mathfrak p$ are the energy and pressure densities introduced in Remark \ref{rem.logconcave}.

\begin{exa}[Standard cosmological models] \label{exa.model}
\begin{enumerate}
\item The space-time $\mathcal M:= (0, +\infty) \times_{\alpha} \mathbb H^3$ where $\alpha(t):=t$ satisfies the condition
 $I_{-}(\alpha)=I_{+}(\alpha) = +\infty$. In that case, we have $\mathfrak p=\mathfrak q=0$, i.e. $\mathcal M$ solves Einstein's Equations in  vacuum.
\item The space-time $\mathcal M:= (0, 2) \times_{\alpha} \mathbb S^3$ where $\alpha(t):=\sqrt{t(2-t)}$ satisfies the conditions
$I_{-}(\alpha)<+\infty$ and $I_{+}(\alpha) <+\infty$. It corresponds to the state equation $\mathfrak p=\mathfrak q /3$, and it models a universe with spherical slices and radiation-dominated.
\item The space-time $\mathcal M:= (0, +\infty) \times_{\alpha} \mathbb R^3$ where $\alpha(t):=t^{2/3}$ satisfies the conditions
$I_{-}(\alpha)<+\infty$ and $I_{+}(\alpha) =+\infty$. It corresponds to the state equation $\mathfrak p=0$ and $\mathfrak q >0$, and it models a spatially flat universe  which is matter-dominated.
\item By choosing an appropriate coordinate system, De Sitter space-time can be written as a RW space-time $\mathcal M:= (0, +\infty) \times_{\alpha} \mathbb H^3$ where $\alpha(t):=\sinh(t)$. Thus, it satisfies the conditions $I_{-}(\alpha)=+\infty$ and $I_{+}(\alpha) <+\infty$. It corresponds to the state equation $\mathfrak p=-\mathfrak q >0$, and it models a universe  which is vacuum-dominated. De Sitter space-time solves Einstein's Equations in vacuum but with a positive cosmological constant. 
\end{enumerate}
\end{exa}

Having identified the structure of the causal boundary, it is possible to give meaning to the notion of convergence towards a point of the boundary, see \cite{flores}. 
Namely, since RW space-times are conformal to Minkowski space-time, and since the notion of causal boundary is conformally invariant, the convergence to the causal boundary in RW space-times can be deduced from the convergence to causal boundary in Minkowski space-time $\mathbb R^{1,d}$. In that case, one can show (see e.g. Section 6.8 and 6.9 of \cite{hawell}) that the future causal boundary can be identified with $\mathbb R^+ \times \mathbb S^{d-1}$, an inextingible future oriented curve $(t_s,x_s) \in \mathbb R^{1,d}$ converging to the boundary if and only if $t_s$ goes to infinity with $s$, $$\theta_s:=\frac{x_s}{|x_s|} \to \theta_{\infty} \in \mathbb S^{d-1}, \;\; \hbox{and} \;\; \delta_s:=t_s-\langle x_s, \theta_{\infty} \rangle \to \delta_{\infty} \in \mathbb R^+,$$
where $\langle \cdot \, , \cdot \rangle$ and $|\cdot |$ denote the Euclidean scalar product and norm in $\mathbb R^{d}$. In other words, the Euclidean curve $x_s \in \mathbb R^{d}$ goes to infinity in the direction $\theta_{\infty} \in \mathbb S^{d-1}$, and the space-time curve $(t_s,x_s)$ goes to infinity in the same direction along the affine hyperlane
$$\Pi(\delta_{\infty}, \theta_{\infty}):=\left \lbrace  (t,x) \in (0,T)\times_{\alpha} \mathbb R^{d}, \;\; t-\langle x, \theta_{\infty} \rangle = \delta_{\infty}\right \rbrace,$$
which is parallel to the tangent hyperplan to the lightcone and containing the direction $\theta_{\infty}$, see Fig. \ref{fig.minko} below.

\begin{figure}[ht]
\hspace{3.5cm}\includegraphics[scale=0.5]{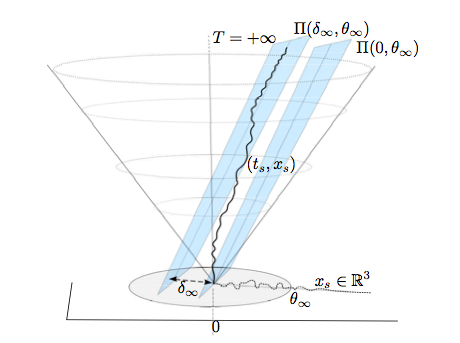}
\caption{Convergence to the causal boundary identified with $\mathbb R^+ \times \mathbb S^2$ in Minkowski space-time.}\label{fig.minko}
\end{figure}
\par
\smallskip 
Performing the time change $$t \to \int_{t_0}^t du /\alpha(u)$$ then allows to deduce the typical behavior of a curve converging to the future causal boudary in RW space-times of the type $\mathbb R^+ \times_{\alpha} \mathbb R^{d}$, and similarly for the space-times of the form $\mathbb R^+ \times_{\alpha} \mathbb H^{d}$ or $\mathbb R^+ \times_{\alpha} \mathbb S^{d}$ after a conformal change of the radial variable in polar coordinates.
Of course, as the boundary itself heavily depends on the choice of the torsion function $\alpha$ and the Riemannian fiber $M$, the resulting notion of convergence differs drastically depending on the type of RW space-time considered. Here are some examples to illustrate this phenomenon. 
\newpage
\begin{exa}[Convergence to the causal boundary] \textcolor{white}{blank}\par\label{exa.cbound}
\begin{enumerate}
\item If $I_{+}(\alpha) <+\infty$, a future oriented timelike curve $(\xi_s)_{s \geq 0}=(t_s,x_s)_{s \geq 0} \in \mathcal M:=(0, T) \times_{\alpha} M$ converges to a point $\xi_{\infty}$ on $\pa \mathcal M_c^+$ iff $t_s \to T$ and $x_s$ converges to a point $x_{\infty}$ in $M$, so that the boundary $\pa \mathcal M_c^+$ can be identified with $\{T\} \times M$, and $\xi_{\infty}$  with $(T,x_{\infty})$.
\item If $M=\mathbb R^{d}$ and  $I_{+}(\alpha) =+\infty$, from the above theorem we have $\pa \mathcal M_c^+\backslash \{ i_{+}\}=\mathbb R^+ \times \mathbb S^{d-1}$. In that case, a future oriented timelike curve $(\xi_s)_{s \geq 0}=(t_s,x_s)_{s \geq 0} \in \mathcal M:=(0, T) \times_{\alpha} \mathbb R^d$ converges to a point $\xi_{\infty}$ on $\pa \mathcal M_c^+ \backslash \{ i_{+}\}$ if and only if 
$$t_s \to T, \quad \theta_s:=\frac{x_s}{|x_s|} \to \theta_{\infty} \in \mathbb S^{d-1}, \;\; \hbox{and} \;\; \delta_s:=\int_{t_0}^{t_s} \frac{du}{\alpha(u)}-\langle x_s, \theta_{\infty} \rangle \to \delta_{\infty} \in \mathbb R^+,$$
where, as above, $\langle \cdot \, , \cdot \rangle$ and $|\cdot |$ denote the Euclidean scalar product and norm in $\mathbb R^{d}$.
In other words, the Euclidean curve $x_s \in \mathbb R^{d}$ goes to infinity in the direction $\theta_{\infty} \in \mathbb S^{d-1}$, and the curve $\xi_s$ goes to infinity in the same direction along the hypersurface 
\[
\Sigma(\delta_{\infty}, \theta_{\infty}):=\left \lbrace  (t,x) \in (0,T)\times_{\alpha} \mathbb R^3, \;\; \int_{t_0}^t  \frac{du}{\alpha(u)}-\langle x, \theta_{\infty} \rangle = \delta_{\infty}\right \rbrace.\]
This hypersurface is parallel to the hypersurface $\Sigma(0, \theta_{\infty})$ containing the curve $(t,\theta_{\infty}\int_{t_0}^t  \frac{du}{\alpha(u)})$ and orthogonal to $\theta_{\infty}$, see Fig. \ref{fig.compactRW} below. In that case, the limiting point $\xi_{\infty}$ on $\pa \mathcal M_c^+$ can be identified with $(T, \delta_{\infty}, 
\theta_{\infty})$. 
\begin{figure}[ht]
\hspace{3cm}\includegraphics[scale=0.5]{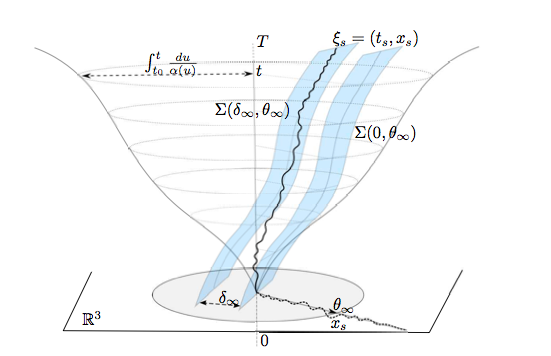}
\caption{Convergence towards the causal boundary in the case $M=\mathbb R^3$ and  $I_{+}(\alpha) =+\infty$.}\label{fig.compactRW}
\end{figure}
\item If $M=\mathbb H^{d}$ viewed as the half-sphere of the Minkowski space $\mathbb R^{1,{d}}$, and if $I_{+}(\alpha) =+\infty$, from the above theorem we have again $\pa \mathcal M_c^+\backslash \{ i_{+}\}=\mathbb R^+ \times \mathbb S^{d-1}$, and this time, a future oriented timelike curve $(\xi_s)_{s \geq 0}=(t_s,x_s)_{s \geq 0} =(t_s,\sqrt{1+r^2_s}, r_s \theta_s)_{s \geq 0} \in \mathcal M:=(0, T) \times_{\alpha} \mathbb H^{d}$ converges to $\xi_{\infty}$ on $\pa \mathcal M_c^+\backslash \{ i_{+}\}$ iff 
\[
t_s \to T, \quad r_s \to +\infty, \quad \theta_s \to \theta_{\infty} \in \mathbb S^{d-1}, \;\; \hbox{and} \;\; \delta_s:=\int_{t_0}^{t_s} \frac{du}{\alpha(u)}-\ash(r_s)  \to \delta_{\infty} \in \mathbb R^+.
\]
\end{enumerate}
\end{exa}
\subsection{The relativistic diffusion in Robertson-Walker spacetimes}\label{sec.RWdiffusion}
Let us now describe the stochastic process introduced in \cite{flj}, which we will call the \emph{relativistic diffusion} in the sequel, and which generalizes Dudley's diffusion to the realm of General Relativity. Consider a general Lorentzian manifold $(\mathcal M, g)$ of dimension $d+1$, i.e. a differentiable manifold equipped with a  pseudo-metric of signature $(-, +, \ldots, +)$. The relativistic diffusion is a diffusion process that takes values in the tangent bundle of $(\mathcal M, g)$. Its sample paths $(\xi_s, \dot{\xi}_s)$ are time-like curves that are future directed and parametrized by the arc length $s$ so that the diffusion actually live on the positive part of the unitary tangent bundle of the manifold, that we simply denote by $T^1_+ \mathcal M$.
The infinitesimal generator of the diffusion can be written 
\begin{equation}\mathcal L:= \mathcal L_0 + \frac{\sigma^2}{2 } \Delta_{\mathcal V},\end{equation}
where the differential operator $\mathcal L_0$ generates the geodesic flow on $T^1 \mathcal M$, $\Delta_{\mathcal V}$ is the vertical Laplacian, and $\sigma > 0$ is a real parameter. Equivalently, if $\xi^{\mu}$ is a local chart on $\mathcal M$ and if $\Gamma_{\nu \rho}^{\mu}$ denote the usual Christoffel symbols, the relativistic diffusion is the solution of the following system of stochastic differential equations (in It\^o form), for $0 \leq \mu \leq d=\mathrm{dim} (\mathcal M)$ :
\begin{equation}\label{eqn.flj}
\left \lbrace \begin{array}{l}
\displaystyle{ d \xi^{\mu}_s = \dot{\xi}_s^{\mu} ds}, \\
\displaystyle{ d \dot{\xi}^{\mu}_s= -\Gamma_{\nu \rho}^{\mu}(\xi_s)\, \dot{\xi}^{\nu}_s \dot{\xi}^{\rho}_s ds + d \times \frac{\sigma^2}{2}\, \dot{\xi}^{\mu}_s ds+ \sigma  d M^{\mu}_s}, 
\end{array}\right.
\end{equation}
where the braket of the martingales $M^{\mu}_s$ is given by
$$\langle dM_s^{\mu}, \;  dM_s^{\nu}\rangle = (\dot{\xi}^{\mu}_s \dot{\xi}^{\nu}_s +g^{\mu \nu}(\xi_s))ds.$$
Moreover, since the sample paths are parametrized by the arc length $s$, we have the pseudo-norm relation:
\begin{equation}\label{eqn.pseudo}
g_{\mu \nu}(\xi_s) \dot{\xi}^{\mu}_s \dot{\xi}^{\nu}_s =-1.
\end{equation}

\begin{rem}In the limit case when the parameter $\sigma$ is choosen to be zero, Equations (\ref{eqn.flj}) are nothing but the geodesics equations: 
$$ \displaystyle{ \frac{d \xi^{\mu}_s}{ds} = \dot{\xi}_s^{\mu} },  \quad
\displaystyle{ \frac{d\dot{\xi}^{\mu}_s}{ds}= -\Gamma_{\nu \rho}^{\mu}(\xi_s)\, \dot{\xi}^{\nu}_s \dot{\xi}^{\rho}_s },$$
so that, the sample paths of the relativistic diffusion can really be thought as random perturbations of timelike geodesics.
\end{rem}

For example, in the case of a spatially flat RW space-time $\mathcal M=I\times_{\alpha} \mathbb R^{d}$, equipped with the canonical global coordinates $(\xi^0, \xi^1, \ldots, \xi^{d})=(t, x^1, \ldots, x^{d})$, the metric is $g_{\mu \nu} = \mathrm{diag}(-1, \alpha^2(t), \ldots, \alpha^2(t))$, and the only non vanishing Christoffel symbols are $\Gamma_{i\,i}^0 = \alpha(t) \alpha'(t)$, and $\Gamma_{0\, i}^i = H(t)$ for $i=1, \ldots, {d}$. 
Thus, if $|\dot{x}_s|$ denote the usual Euclidean norm of $\dot{x}_s$ in $\mathbb R^{d}$,  Equations (\ref{eqn.flj}) simply reads 
\begin{equation}\label{eqn.flj.eucli}
\left \lbrace \begin{array}{ll}
\displaystyle{d t_s=\dot{t}_s ds}, & \quad \displaystyle{d \dot{t}_s = - \alpha(t_s) \: \alpha'(t_s) |\dot{x}_s|^2 ds + \frac{{d} \sigma^2}{2} \dot{t}_s ds + \sigma d M^{\dot{t}}_s,} \\
\\
\displaystyle{d x^i_s = \dot{x}_s^i ds}, &  \quad \displaystyle{d \dot{x}^i_s = \left(-2 H(t_s) \dot{t}_s + \frac{{d} \sigma^2}{2}\right)  \dot{x}^i_s\, ds + \sigma dM^{\dot{x}^i}_s},
\end{array}\right. 
\end{equation}
where
$$
\left \lbrace \begin{array}{l}
\displaystyle{d \langle M^{\dot{t}}, \, M^{\dot{t}} \rangle_s =  \left( \dot{t}_s^2-1 \right) ds, \quad d \langle M^{\dot{t}}, \, M^{\dot{x}^i} \rangle_s =  \dot{t}_s \dot{x}^i_s ds,} \\
\\
\displaystyle{d \langle M^{\dot{x}^i}, \, M^{\dot{x}^j} \rangle_s = \left(\dot{x}^i_s \dot{x}^j_s + \frac{\delta_{ij}}{\alpha^2(t_s)}  \right) ds}.
\end{array}\right.$$
In the case of a RW space-time, the pseudo-norm relation (\ref{eqn.pseudo}) can be written  
\begin{equation}
\label{eqn.pseudorw}-\dot{t}_s^2+\alpha^2(t_s) h(\dot{x}_s, \dot{x}_s)=-1.
\end{equation}

\section{Statement of the results}\label{sec.results}

Having introduced the geometric and probabilistic backgrounds, we can now state our results concerning the long-time asymptotic behavior of the relativistic diffusion on RW space-times.  When non elementary, the proofs of these results are postponed in Section \ref{sec.proofs}.


\subsection{Existence, uniqueness, reduction of the dimension}\label{sec.exiuni}

The following proposition ensures that, in the case of a RW space-time, the system of stochastic differential equations (\ref{eqn.flj}) admits a unique solution, and it exhibits a lower dimensional sub-diffusion that will facilitate its study. 
 
\begin{prop}\label{pro.exiuni}
Let $\mathcal M:=(0, T) \times_{\alpha} M$ be a RW space-time satisfying the hypotheses of Section \ref{sec.RW}.
For any $(\xi_0, \dot{\xi}_0)=(t_0, x_0, \dot{t}_0, \dot{x}_0) \in T^1_+ \mathcal M$, Equations (\ref{eqn.flj}) and (\ref{eqn.pseudo})
admit a unique strong solution $(\xi_s, \dot{\xi}_s)=(t_s, x_s, \dot{t}_s, \dot{x}_s)$ starting from $(\xi_0, \dot{\xi}_0)$, which is well defined up to the explosion time $\tau:=\inf \{ s>0, \, t_s= T\}$. If $T<+\infty$, this explosion time is finite almost surely whereas if $T=+\infty$, $\ta$ is almost surely infinite. Moreover, the temporal process $(t_s, \dot{t}_s)_{s \geq 0}$ is a two-dimensional  sub-diffusion. 
\end{prop}

\begin{rem}
In the sequel, given a point $(\xi_0, \dot{\xi}_0)\in T^1 \mathcal M$, the unique solution $(\xi_s, \dot{\xi}_s)_{0\leq s <\tau }$ of Equations (\ref{eqn.flj}) and (\ref{eqn.pseudo}) will be called the relativistic diffusion starting from $(\xi_0, \dot{\xi}_0)$. We will denote by $\mathbb P_{0}$ its law; unless otherwise stated, the words ``almost surely'' will mean $\mathbb P_0-$almost surely. 
\end{rem}

\begin{proof}
Recall that from Hypotheses 1, the torsion function $\alpha$ is positive and of class $C^2$ on the interval $(0, T)$. Therefore, on this interval, the metric $g$ and its inverse, the Christoffel symbols $\Gamma$, hence the coefficients of Equations (\ref{eqn.flj}) are smooth functions and classical results ensure existence and uniqueness of the solution until explosion\footnote{Actually, the diffusion coefficient of $\dot{t}_s$ is only $1/2-$H\"older but a simple change of variable leads to locally Lipschitz coefficients, hence classical results mentionned above apply, see Lemma \ref{LEM.UNIEXITEMP} and its proof below.} (see for example Theorem 2.3 p. 173 of \cite{ikeda} and Theorem 1.1.9 of \cite{hsu}). Making explicit the Christoffel symbols, the first two equations of (\ref{eqn.flj})  read 
$$\left \lbrace \begin{array}{ll} \displaystyle{d t_s=\dot{t}_s ds},  \\ \displaystyle{d \dot{t}_s = - \alpha(t_s) \: \alpha'(t_s) h(\dot{x}_s, \dot{x}_s)  ds + \frac{{d} \sigma^2}{2} \dot{t}_s ds + \sigma d M^{\dot{t}}_s,} \end{array} \right. \;\; \hbox{with} \;\;  d \langle M^{\dot{t}}, \, M^{\dot{t}} \rangle_s =  \left( \dot{t}_s^2-1 \right) ds.$$
Using the pseudo-norm relation (\ref{eqn.pseudorw}), we have $h(\dot{x}_s, \dot{x}_s) = \alpha^{-2}(t_s) \left( \dot{t}_s^2-1 \right)$ so that
$$d \dot{t}_s = - H(t_s) \left( \dot{t}_s^2-1 \right) ds + \frac{{d} \sigma^2}{2} \dot{t}_s ds + d M^{\dot{t}}_s,$$
and the couple $(t_s, \dot{t}_s)$ is indeed a sub-diffusion of dimension two of the whole process $(\xi_s, \dot{\xi}_s)$.
From the pseudo-norm relation again, we have $\alpha^2(t_s) h(\dot{x}_s, \dot{x}_s)  \leq \dot{t}_s^2$ so that the spatial components
 $(x_s, \dot{x}_s)$ cannot explode before $\dot{t}_s$, at least if $t_s < T$. Thus, the assertion on the lifetime of the global diffusion is a consequence of Lemmas \ref{LEM.UNIEXITEMP} and \ref{LEM.LIFETIME} below  which show in particular that the lifetime of the temporal process $(t_s, \dot{t}_s)_{s \geq 0}$ is almost surely infinite in the case $T=+\infty$ whereas it is almost surely finite in the case $T<+\infty$.
\end{proof}

\subsection{Entrance law at the origin of time}\label{sec.entree}
Since the geometry of space time at the ``origin of time'' $t=0$  is of crucial importance in cosmology, it is natural to investigate if the relativistic diffusion can be started from a point of the form $(t_0=0, x_0, \dot{t}_0, \dot{x}_0)$ living on an ``entrance boundary'' of  $T^1\mathcal M$ where  $\mathcal M=(0,T) \times_{\alpha} M$. Of course, the natural candidate for this entrance boundary is the causal boundary $\pa \MM_c^{-}$ corresponding to past oriented causal curves. From Section \ref{sec.RW}, one knows that the nature of $\pa \MM_c^{-}$ depends on the finiteness of $I_{-}(\alpha)$. \par
\smallskip
In the case $I_{-}(\alpha)=+\infty$, the boundary $\pa \MM_c^{-}$ is ``really'' a boundary at infinity, in particular it has no a priori tangent structure and even in the deterministic case i.e. considering the geodesic flow on $T^1 \mathcal M$, one can show that the data of a point on $\pa \MM_c^{-}$ is not sufficient to uniquely determine a time-like geodesic curve in $\mathcal M$.
On the contrary, in the case $I_{-}(\alpha)<+\infty$, the boundary $\pa \MM_c^{-}$ identifies with $\{ 0\} \times M$ and inherits the tangent structure of $M$. In particular, given a point $(x_0, \dot{x}_0) \in T^1 M$, there exists a unique time-like and past oriented geodesic curve with constant normalized derivative $\dot{x}_0$, and which converges to  $(0,x_0) \in \pa \MM_c^{-}$, see Sect. III.4.3 of \cite{mathese}. 

\newpage
The relativistic diffusion being a perturbation of the geodesic flow, the above deterministic considerations indicate that the right framework to consider to define a process starting from the entrance boundary is the one where $I_{-}(\alpha)<+\infty$, i.e. when geodesics have finite past horizon. In that case, one can indeed show that Equations (\ref{eqn.flj}) and (\ref{eqn.pseudo}) admit a solution starting from $t_0=0$ provided $\dot{t}_0=+\infty$ or more precisely provided the product $\alpha(t_0)^2 (\dot{t}_0^2-1)$ is positive and finite. So let us consider the new variables
\[
t \in (0,T), \quad a= \alpha(t) \sqrt{\dot{t}^2-1} \in (0,+\infty) , \quad x \in M, \quad \hbox{and}\;\;  \dot{x}/|\dot{x}| \in T_x^1 M,
\]
where $|\dot{x}|$ denote the square root of $h(\dot{x},\dot{x})$ to simplify the expressions. Note that, from Equation (\ref{eqn.pseudorw}), the norm $|\dot{x}|$ can be written as a function of $t$ and $a$ so that $(t,a,x,\dot{x}/|\dot{x}|)$ is a coordinate system on $T^1 \mathcal M$. 

\begin{prop}\label{pro.entreeglobal}
Fix $t_0=0$, $a_0>0$ and $(x_0,\dot{x}_0/|\dot{x}_0|) \in T^1 M$. Then, in the above coordinate system, the stochastic differential equation system (\ref{eqn.flj}) admits a unique strong solution $(t_s, a_s, x_s, \dot{x}_s/|\dot{x}_s|)$ starting from $(t_0, a_0, x_0, \dot{x}_0/|\dot{x}_0|)$ and which is well defined up to the explosion time $\ta=\inf \{s>0, \; t_s = T\}$.
\end{prop}
The proof of Proposition \ref{pro.entreeglobal} goes into two steps: first show existence and uniqueness for the temporal process starting form $t_0=0$ and $a_0>0$, and then deduce existence and uniqueness for the spatial components. It can be found in \cite{mathese}, see Propositions IV.5 and IV.6 of Sect. IV.4.
\begin{rem}
If $T<+\infty$ and both $I_{-}(\alpha)$ and $I_+(\alpha)$ are finite, the manifold $\mathcal M:=(0, T) \times_{\alpha} M$ is geodesically incomplete and the lifetime of the relativistic diffusion is finite almost surely. Nevertheless, different copies of $\mathcal M$ can be concatenated by identifying the exit boundary of the one and the entrance boundary of the other (a Big-Bang follow a Big-Crunch and so on...) to form a ``pearl necklace'' of RW space-time which can then be identified with $\mathbb R^+ \times_{\alpha} M$. In this setting, one can show that excursions of the relativistic diffusion started from $\{0\} \times M$  and ending almost surely on $\{T\} \times M$ can also be concatenated to form a diffusion on the ``necklace" $\mathbb R^+ \times_{\alpha} M$, see Part II, Chap. VII of \cite{mathese} .
\end{rem}


\subsection{Convergence towards the causal boundary} \label{sec.conv}

We can now state our main result concerning the almost sure convergence of the projection of the relativistic diffusion on the base manifold $\mathcal M$ to a random point of its causal boundary $\pa \MM_c^+$.
\begin{thm}\label{theo.causal}
Let $\mathcal M:=(0, T) \times_{\alpha} M$ be a RW space-time satisfying the hypotheses of Sect. \ref{sec.RW} and let $(\xi_s, \dot{\xi}_s)=(t_s, x_s, \dot{t}_s, \dot{x}_s)$ be the relativistic diffusion starting from $(\xi_0, \dot{\xi}_0) \in T^1_+ \MM$. Then almost surely, as $s$ goes to $\ta=\inf \{ s>0, \, t_s= T\}$, the projection $\xi_s$ converges towards a random point of the causal boundary $\pa \MM_c^+$ of $\MM$. 
\end{thm}

\begin{rem}
In fact, we prove a more precise result: if the curvature of the Riemannian fiber $M$ is non-negative, then $\xi_s$ always converges to a random point of $\pa \MM_c^+ \backslash \{i^+\}$, whereas in the compact case $M=\mathbb S^d$, $\xi_s$ converges to a random point of $\pa \MM_c^+ \backslash \{i^+\}$ or to $i^+$ depending on the finiteness of $I_+(\alpha)$.
\end{rem}
\noindent
As already mentionned in the introduction, the geometry of the causal boundary strongly reflects the one of the base manifold so that the above theorem actually covers a huge variety of geometric asymptotic behaviors, depending on the type of RW space-times considered. The proof of Theorem \ref{theo.causal} and explicit examples of convergence are given in Section 4.\par
\begin{exa}
\begin{figure}[ht]
\hspace{1cm}\includegraphics[scale=0.45]{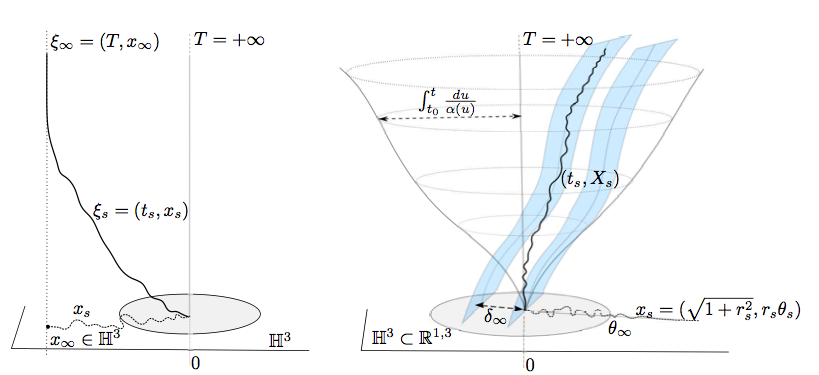}
\caption{Typical behavior of the projection $\xi_s$ in $\MM=]0, +\infty[ \times_{\a} \mathbb H^3$ when $I_+(\alpha) < +\infty$ (left) and $I_+(\alpha) =+\infty$ (right).}
\end{figure}
The figure above represents the typical behaviors of the process $\xi_s=(t_s,x_s)$ in the hyperbolic case 
 $\MM=]0, +\infty[ \times_{\a} \mathbb H^3$, depending on the finiteness of the integral $I_+(\alpha)$. In the case when $I_+(\alpha) < +\infty$ and when $s$ goes to infinity, the projection $x_s$ of $\xi_s$ in $\mathbb H^3$ converges almost surely to a random point $x_{\infty} \in \mathbb H^3$ and $t_s$ goes to infinity, so that the process $\xi_s=(t_s,x_s)$  asymptotically describes a line and converges to the point $(T=+\infty, x_{\infty})$ of the causal boundary $\pa \MM_c^+$ according to the first point of Example \ref{exa.cbound}. On the contrary, in the case when $I_+(\alpha) = +\infty$, the projection $x_s \in \mathbb H^3$ is transient. Namely, using the standard polar decompostion $x_s=(\sqrt{1+r^2_s}, r_s \theta_s)$ in $\mathbb H^3$, $r_s$ goes to infinity with $s$ and $\theta_s$ converges almost surely to a random point $\theta_{\infty} \in \mathbb S^2$. Moreover, the first projection $t_s$ also goes to infinity with $s$ and the process $(t_s, X_s)$ where $X_s :=(\ash(\sqrt{1+r^2_s}), \ash(r_s) \theta_s)$ goes to infinity along a random hypersurface, i.e. the process $\xi_s=(t_s, x_s)$ converges to a random point of $\pa \MM_c^+$ according to the third point Example \ref{exa.cbound}.
 \end{exa}

\begin{rem}
The proof of Theorem \ref{theo.causal} actually shows that not only the sample paths of the process $\xi_s$ converge to random points on the causal boundary but they are in fact asymptotic to random lightlike geodesics. Therefore Theorem \ref{theo.causal} goes in the direction of \cite{flj}'s statement that for a general Lorentz manifold, the Poisson boundary i.e. the set of bounded harmonic functions can be characterized by classes of light rays i.e. null geodesics.
\end{rem}


\subsection{Asymptotic behavior of the derivative} \label{sec.derivative}

In order to give a complete picture of the almost sure asymptotic behavior of the whole relativistic diffusion  $(\xi_s, \dot{\xi}_s)$ in RW space-times, we now specify the asymptotic behavior of $\dot{\xi}_s \in T^1_{\xi_s}$.

\subsubsection{Asymptotics of the temporal components}
\smallskip
The next theorem shows that the asymptotic behavior of the temporal derivative $\dot{t}_s$ is governed by the growth of the torsion function $\alpha$ at the endpoint $T$. Recall that in the case $T=+\infty$, the following limit $H_{\infty}:=\lim_{t \to +\infty} H(t)=\lim_{t \to +\infty}  \alpha'(t)/\alpha(t)$, exists and is non-negative. 

\begin{thm}\label{theo.causalderivetemp}
Let $\mathcal M:=(0, T) \times_{\alpha} M$ be a RW space-time satisfying the hypotheses of Sect. \ref{sec.RW} and let $(\xi_s, \dot{\xi}_s)=(t_s, x_s, \dot{t}_s, \dot{x}_s)$ be the relativistic diffusion starting from $(\xi_0, \dot{\xi}_0) \in T^1_+ \MM$. Then, when $s$ goes to $\ta=\inf \{ s>0, \, t_s= T\}$, we have the following asymptotic behaviors: 
\begin{enumerate}[$i)$]
\item if $T<+\ii$, $\dot{t}_s$ is almost surely transient;
\item if $T=+\ii$, $H_{\ii}=0$ and the expansion is at most polynomial, $\dot{t}_s$ is almost surely transient;
\item if $T=+\ii$, $H_{\ii}=0$, and the expansion is subexponential, $\dot{t}_s$ converges to $+\infty$ in probability;	
\item if $T=+\ii$ and $H_{\ii}>0$, $\dot{t}_s$ is Harris recurrent in $(1, +\ii)$ almost surely;
\end{enumerate}
\end{thm}

The proof of Theorem \ref{theo.causalderivetemp} is given in Section \ref{sec.temp} below. More precisely, the recurrence property in the case when $H_{\infty}>0$ is proved in Proposition \ref{PRO.ERGO}, whereas the transient cases are treated in Propositions \ref{PRO.CONVA} and  \ref{PRO.TRANSPROB} and corollary \ref{cor.trans.poly} respectively. \par
\smallskip

In the case where $T=+\ii$, $H_{\ii}=0$ and the expansion is subexponential, Theorem \ref{theo.causalderivetemp} only states a convergence in probability and it is tempting to ask if this convergence holds almost surely or not. The next proposition gives a quite surprising necessary and sufficient condition in term of the rate of decrease of the Hubble function.

\begin{prop} \label{pro.H3ssi}If the growth rate of the torsion function is subexponential, then the process $\dot{t}_s$ goes almost surely to infinity with $s$ if and only if the function $H^d$ is integrable at infinity.
\end{prop}
\begin{exa}
In dimension $d=3$, if the warping function is of the form $\alpha(t)=\exp \left(t^{\beta}\right)$ with $\beta \in (0,1)$, then $H(t)= \beta t^{\beta-1}$ and the last proposition ensures that $\dot{t}_s$ is almost surely transient if and only if $\beta < 2/3$.
\end{exa}

The above qualitative behaviors can be made more explicit with quantitative estimates. In particular, in the transient case, we relate the almost-sure speed of divergence of $\dot{t}_s$ to the growth rate of the torsion function, see Sect. \ref{SEC.ASYMPTEMPTRANS}.

\subsubsection{Asymptotics of the spatial components}
\smallskip
We conclude this section by expliciting  the asymptotic behavior of the spatial derivative $\dot{x}_s$. Since the squared norm $|\dot{x}_s|^2:=h(\dot{x}_s,\dot{x}_s)$ in  $T_{x_s} M$ is related to the temporal process via the pseudo-norm relation (\ref{eqn.pseudorw}), we are here more particularly interested in the normalized derivative  $\dot{x}_s/|\dot{x}_s|$. Once again, we have a dichotomy depending on the growth of the torsion function and the finiteness of the integral $I_+(\alpha)$. Before stating our results, let us sketch a picture of the situation.
\par
\smallskip
If $I_+(\alpha)<+\infty$, Theorem \ref{theo.causal} and the first point of Example \ref{exa.cbound} ensure that the projection $x_s$ converges almost surely to a random point $x_{\ii} \in M$ (see Proposition \ref{pro.horizon} below).  If $T<+\infty$, Theorem \ref{theo.causalderive} below shows that $\dot{x}_s/|\dot{x}_s|$ is always convergent, see figure \ref{fig.transrec} (left) below.
But if $T=+\infty$, as in the case of the temporal process, we show that the recurrence/transience of $\dot{x}_s/|\dot{x}_s|$ is governed by the type of expansion considered. Namely, if rate of decrease of $H$ is fast enough i.e if $\alpha$ is of polynomial growth or of subexponential growth with $H^3 \in \mathbb L^1$, then $\dot{x}_s/|\dot{x}_s|$ converges almost surely to a random point $\Theta_{\ii} \in T_{x_{\ii}}^1 M$, whereas if $H_{\infty}>0$ or if $\alpha$ is of subexponential growth with $H^3 \notin \mathbb L^1$, $\dot{x}_s/|\dot{x}_s|$ asymptotically describes a recurrent time-changed spherical Brownian motion in the limit unitary tangent space $T_{x_{\ii}}^1 M \approx \S^2 $. In the latter case, the projection $x_s$ is thus convergent but the convergence to the limit random point $x_{\infty}$ is very irregular, see Fig. \ref{fig.transrec} (right).

\begin{figure}[ht]
\includegraphics[scale=0.5]{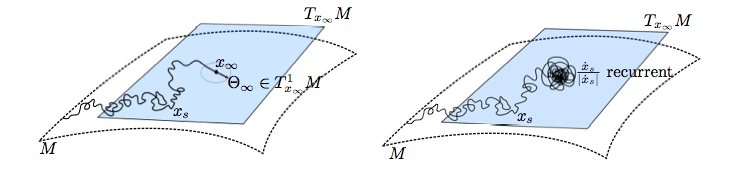}
\caption{Left: typical behavior of the projection $x_s$  in $M$ when $T<+\infty$ or when $T=+\infty$, $I_+(\alpha)<+\infty$ and the expansion is polynomial; Right: typical behavior of the projection $x_s$  in $M$ when $T=+\infty$, $I_+(\alpha)<+\infty$ and the expansion is exponential.\label{fig.transrec}}
\end{figure}

\newpage
\noindent
\begin{thm}\label{theo.causalderive}
Let $\mathcal M:=(0, T) \times_{\alpha} M$ be a RW space-time satisfying the hypotheses of Sect. \ref{sec.RW} and such that $I_+(\a)<+\infty$.
Let $(\xi_s, \dot{\xi}_s)=(t_s, x_s, \dot{t}_s, \dot{x}_s)$ be the relativistic diffusion starting from $(\xi_0, \dot{\xi}_0) \in T^1_+ \MM$. Then almost surely, when $s$ goes to $\ta=\inf \{ s>0, \, t_s= T\}$, the projection $x_s$ converges to a random point $x_{\ii} \in M$ and the normalized derivative $\dot{x}_s/|\dot{x}_s|$ satisfies:  
\begin{enumerate}[$i)$]
\item if $T<\ii$, then $\dot{x}_s/|\dot{x}_s|$ converges to a random point $\Theta_{\ii} \in T_{x_{\ii}}^1 M$;
\item if $T=+\ii$ and the growth rate of $\alpha$ is at most polynomial then $\dot{x}_s/|\dot{x}_s|$ converges to a random point $\Theta_{\ii} \in T_{x_{\ii}}^1 M$;
\item if $T=+\ii$ and the growth rate of $\alpha$ is subexponential with  with $H^3 \in \mathbb L^1$ if $d>4$ or $H^3 \in \mathbb L^{1^-}$ if $d=3$, then $\dot{x}_s/|\dot{x}_s|$ converges to a random point $\Theta_{\ii} \in T_{x_{\ii}}^1 M$;
\item if $T=+\ii$ and $\alpha$ is of exponential growth or is of subexponential growth with $H^3 \notin \mathbb L^1$, then $\dot{x}_s/|\dot{x}_s|$ is recurrent. More precisely, if $M=\mathbb R^{d}$, the process $\dot{x}_s/|\dot{x}_s|$ is a recurrent time-changed spherical Brownian motion. In the case $M=\mathbb H^{d} \subset \mathbb R^{d+1}$ or $M=\mathbb S^{d} \subset \mathbb R^{d+1}$, for all $\e>0$, there exists a proper time $s_{\e}$ that is almost surely finite and a recurrent time-changed spherical Brownian motion $(\Theta_s^{\e}, \, s \geq s_{\e})$ in $T_{x_{\ii}}^1 M \approx \S^{d-1} $ such that: 
$$\sup_{s \geq s_{\e}}\left|\left| \frac{\dot{x}_{s }}{|\dot{x}_{s}|}- \Theta_s^{\e}\right| \right|\leq \e,$$
where $|| \cdot ||$ denotes the Euclidean norm in the ambiant space $\mathbb R^{d+1}$.
\end{enumerate}
\end{thm}

\noindent
The proof of Theorem \ref{theo.causalderive} is given in Sect. \ref{sec.tangentunitaire}. 

\begin{rem}
In the case $I_+(\a)<+\infty$ i.e. when the manifold $\mathcal M$ has finite horizon, it is interesting to note that whatever the nature of the Riemannian manifold $M$, the process $\dot{x}_s/|\dot{x}_s|$ is either convergent or it asymptotically describes a  recurrent time-changed spherical Brownian motion in the limit unitary tangent space. In other words, the normalized spatial derivative does not ``see'' the curvature of the Riemannian manifold $M$, its asymptotic behavior only depends on the torsion function $\a$.
\end{rem}

Finally, we describe the asymptotic behavior of the normalized derivative in the case when $I_+(\alpha)=+\infty$. When the Riemannian manifold $M$ has non positive curvature, and when properly rescaled, the process $\dot{x}_s$ is shown to be convergent. On the contrary, in the spherical case, the process $\dot{x}_s/|\dot{x}_s|$ has a remarkable asymptotic behavior: almost surely, it asymptotically describes a random great circle on the $d-$dimensional Euclidean sphere.

\begin{thm} \label{theo.cercasymp}
Let $\MM=(0, T) \times_{\a} M$ be a RW space-time satisfying the hypotheses of Sect. \ref{sec.RW} and such that $I_+(\a)=+\infty$. Let $(\xi_s, \dot{\xi}_s)=(t_s, x_s, \dot{t}_s, \dot{x}_s)$ be the relativistic diffusion starting from $(\xi_0, \dot{\xi}_0) \in T^1_+ \MM$.  Then the process $x_s$ is transient and its normalized derivative $\dot{x}_s/|\dot{x}_s|$ satisfies:  
\begin{enumerate}[$i)$]
\item if $M=\R^{d}$, then $\dot{x}_s/|\dot{x}_s|$ converges to a random point $\Theta_{\ii} \in \S^{d-1}$ ; 
\item if $M=\H^{d} \subset \R^{1,{d}}$, then $|x^0_s|^{-1} \times \dot{x}_s/|\dot{x}_s|$ converges to a random point $(1,\Theta_{\ii})$ ;
\item if $M=\S^{d} \subset \R^{{d+1}}$, then both $x_s$ and $\dot{x}_s/|\dot{x}_s|$ asymptotically describe a random great circle in the Euclidean sphere $\S^{d}$.
\end{enumerate}
\end{thm}

\begin{figure}[ht]
\hspace{4cm}\includegraphics[scale=0.6]{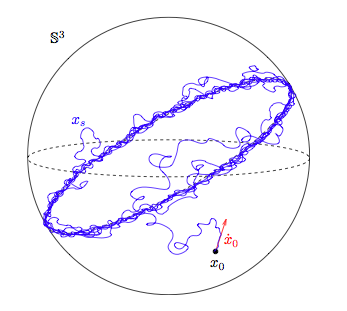}
\caption{Typical behavior of $x_s$ and $\dot{x}_s/|\dot{x}_s|$ in $\mathbb S^3 \subset \mathbb R^4$  in the case $I_+(\alpha)=+\infty$.}
\end{figure}

The proof of Theorem \ref{theo.cercasymp} is given in Sect. \ref{sec.cercasymp}.

\section{Proofs of the results}\label{sec.proofs}

We now give the proofs of the results stated above. The study of the long-time behavior of the temporal diffusion requires a certain amount of work, see Sect. \ref{sec.uniexi}--\ref{sec.temp}, particularly due the fact that the dynamics of $\dot{t}_s$ is really inhomogeneous in the sense that it depends drastically on $t_s$ via the Hubble function $H$. In Sect. \ref{sec.spatial}, we then give the proofs of the results concerning the spatial components of the diffusion: roughly speaking, $(x_s, \dot{x}_s)$ can be seen as an inhomogeneous diffusion on $T^1 M$, parametrized by a clock which depends only on the temporal process. The situation here is very similar to the one of a Brownian motion on a rotationally invariant Riemannian manifold seen in polar coordinates, where the angle is spherical Brownian motion parametrized by an additive functional of the radial process. In our Lorentzian setting, the two-dimensional temporal process plays the role  of the radial process and the spatial process plays the role of the angular component of the Riemannian case.

\subsection{Existence, uniqueness and lifetime of the temporal process} \label{sec.uniexi}
Let us first show existence, uniqueness of the temporal process and explicit its lifetime. 

 From the proof of Proposition \ref{pro.exiuni}, the temporal process is solution to the following system of stochastic differential equations: 
\begin{equation}\label{eqn.ttpoint}
\left \lbrace \begin{array}{ll}
\ds{ d t_s = \dot{t}_s ds, } \\ 
\ds{d \dot{t}_s= -H(t_s) \times \left( \dot{t}_s^2-1 \right) ds  + \frac{d \s^2}{2} \dot{t}_s ds+ \s d M^{\dot{t}}_s,} \end{array} \right. \;\; \hbox{with} \;\;  d \langle M^{\dot{t}}, \, M^{\dot{t}} \rangle_s =  \left( \dot{t}_s^2-1 \right) ds.
\end{equation}

\begin{lemma}\label{LEM.UNIEXITEMP}
For any starting point $(t_0,\dot{t}_0) \in (0,T) \times [1, +\infty)$, Equation (\ref{eqn.ttpoint}) admits a unique strong solution $(t_s,\dot{t}_s)$, which is well defined up to the explosion time $\ta:=\inf \{s>0, t_s = T\}$, and such that $\dot{t}_s >1$ almost surely for all $0<s < \ta$. 
\end{lemma}

\begin{proof}

The coefficients in Equation (\ref{eqn.ttpoint}) being continuous functions of  $(t, \dot{t})$, by classical existence results (e.g. Theorem (2.3) p. 173 of \cite{ikeda}), it admits a strong solution up to explosion. Note that the diffusion coefficient in Equation (\ref{eqn.ttpoint}) is only $1/2-$H\"older in the neighborhood of $\dot{t}=1$. Nevertheless, we can consider the change of variable $(t_s, \dot{t}_s) \to (t_s, a_s^2)$ where 
\begin{equation}\label{def.a}
a_s:=\a(t_s) \sqrt{\dot{t}_s^2-1}.
\end{equation}
By It\^o's formula, we have 
\begin{equation} \left \lbrace \begin{array}{l}
\ds{ d t_s = \sqrt{1 + \frac{a^2_s}{\a^2(t_s)}} \, ds,} \\ \\
\ds{d a^2_s = (d+1)\s^2 a^2_s ds +d \s^2  \a^2(t_s) ds + d M^{a^2}_s,}   
\end{array}\right.\label{eqn.tadeux}\end{equation}
where $d \langle M^{a^2}\rangle_s = 4 \s^2 a^2_s \left( a^2_s + \a^2(t_s)\right) ds$. 
The coefficients in Equation (\ref{eqn.tadeux}) are now locally Lipschitz functions of $(t,a^2)$ in $(0, T) \times [0, +\ii[$, hence classical theorems (e.g. Theorem 1.1.9 of \cite{hsu}) ensure existence and unicity up to the explosion time $\ta \wedge \ta' \wedge \tau''$, where
\[
 \ta:= \inf \{ s>0, \, t_s = T\}, \quad \ta':= \inf \{ s>0, \, a_s = +\ii\}, \quad \hbox{and}  \quad \tau'':=\inf \{s >0, \, a_s=0 \}.
 \]
In fact, we have $\ta \leq \ta' \wedge \tau''$ almost surely. Indeed,  fix an integer $n_0$ such that $t_0 \leq T-1/n_0$ and consider the random times $\ta_n:= \inf \{s>0, \, t_s \geq n \wedge (T-1/n)\}$ for all  $n \geq n_0$. The coefficients in Equation (\ref{eqn.tadeux}) have linear growth in $a^2_s$ on $[0, \ta_n \wedge \tau''[$. Therefore, Proposition 1.1.11 of \cite{hsu} ensures that, almost surely, $a_s^2$ does not explode before $\ta_n \wedge \tau''$. In other words, we have $\ta_n \wedge \tau'' \leq \ta' \; a.s.$ for all $n \geq n_0$. Letting $n$ go to infinity, we get $\ta \wedge \tau'' \leq \ta' $ almost surely.
We are left to show that $\tau'' \geq \tau$ i.e. $a_s>0$ for $0< s< \ta$ or equivalently $\dot{t}_s >1$. If $\dot{t}_0=1$, one easily checks that $\dot{t}_s>1$ for arbitrary small times $s$, thus without loss of generality, we can suppose that $\dot{t}_0>1$ i.e. $a_0>0$. In that case, the time $\ta''=\inf \{s >0, \: \dot{t}_s=1 \}$ is almost surely positive. Moreover, there exists two linear independant Brownian motions $B$ and $B'$ such that
$$ d a_s = \frac{d \s^2}{2} a_s ds + \frac{(d-1)}{2} \s^2\frac{\a^2(t_s)}{a_s} ds + \s  a_s d B_s +  \s \a(t_s) d B_s'.$$
Applying It\^o's formula, we have for $0 \leq s < \tau \wedge \ta''$: 
$$a_s = a_0 \exp \left( \frac{(d-1)}{2} \s^2s +\s B_s  \right) \exp \left( \frac{(d-2)}{2} \s^2\int_0^s \frac{\a^2(t_u)}{a^2_u}du + B'' \left(\s^2 \int_0^s \frac{\a^2(t_u)}{a^2_u}du  \right)   \right).$$
The first exponential can not go to zero in finite time. Moreover, the last exponential is either positive or goes to plus infinity depending on the finiteness of the integral  $\int_0^s \frac{\a^2(t_u)}{a^2_u}du$, in particular it can not go to zero. Thus, we deduce that $\tau'' \geq \tau$ almost surely, hence the result.  
\end{proof}

\begin{lemma}\label{LEM.LIFETIME}
The explosion time $\ta:=\inf \{s>0, t_s = T\}$ of the temporal process $(t_s, \dot{t}_s)_{s \geq 0}$ is almost surely infinite in the case $T=+\infty$ whereas it is almost surely finite in the case $T<+\infty$.
\end{lemma}

\begin{proof}
From Equation (\ref{eqn.pseudorw}), we have $\dot{t}_s \geq 1$ for all $0\leq s < \tau$ almost surely. In particular, $t_s \geq t_0 +s$  for all $0\leq s < \tau$ so that $\tau$ is necessarily almost surely finite when $T<+\infty$. In the case $T=+\infty$, classical comparison results show that the solution $\dot{t}_s$ of Equation (\ref{eqn.ttpoint}) is bounded above by its analogue in the case where $H \equiv 0$ (see Lemma \ref{LEM.COMPARPROJ} below). The lifetime of this process is almost surely infinite (see Lemma \ref{LEM.CASSTAT} below), hence the result.
\end{proof}

\subsection{Asymptotic behavior of the temporal sub-diffusion} \label{sec.temp}
In this section, we determine the almost sure asymptotic behavior of the temporal process $(t_s, \dot{t}_s)$ i.e. we give the proof of Theorem \ref{theo.causalderivetemp} and Proposition \ref{pro.H3ssi}. From Equation (\ref{eqn.ttpoint}),
there exists a real standard Brownian motion $B$ such that $(t_s, \dot{t}_s)$ is solution to 
\begin{equation}\label{eqn.temp} \left \lbrace \begin{array}{ll}
\ds{d t_s = \dot{t}_s ds, } \\
\ds{d \dot{t}_s= -H(t_s) \times \left( \dot{t}_s^2-1 \right)ds  + \frac{d \s^2}{2} \dot{t}_s ds+ \s \sqrt{\dot{t}_s^2-1} \,dB_s}. \end{array} \right.
\end{equation}

\subsubsection{Transience in the case $T<+\infty$}  \label{sec.tfini}
\smallskip
We establish here the first point of Theorem \ref{theo.causalderivetemp}, i.e. the almost sure transience of $\dot{t}_s$ in the case where $T<+\infty$.

\begin{prop}\label{PRO.CONVA} Suppose that $T<+\infty$, fix $(t_0, \dot{t}_0) \in (0,T) \times [1, +\infty)$ let $(t_s, \dot{t}_s)$ be the solution of Equation (\ref{eqn.temp}) starting from $(t_0, \dot{t}_0)$. Then, when $s$ goes to $\tau:=\inf \{ s>0, \, t_s= T\}$, $\dot{t}_s$ tends to infinity almost surely. More precisely, the process $a_s$ defined by Equation (\ref{def.a}) converges almost surely to a random variable $a_{\ii}$ which is positive and finite almost surely. 
\end{prop}

\begin{proof}
From Equation (\ref{eqn.temp}), applying It\^o's formula, we get for all $0<s_0<s<\ta$:
\begin{equation}\log \left( \frac{\dot{t}_s^2-1}{\dot{t}_{s_0}^2-1}\right) =\log \left( \frac{\a^2(t_{s_0})}{\a^2(t_s)}\right)
+ \s^2 \int_{s_0}^s\frac{(d-1) \dot{t}_u^2-1}{\dot{t}_u^2-1} du+ 2 \s \int_{s_0}^s\frac{\dot{t}_u}{\sqrt{\dot{t}_u^2-1}}dB_u, \label{logtpointcarremoinsun}\end{equation}
or equivalently
\begin{equation}\log \left( \frac{a^2_s}{a^2_{s_0}}\right) = \s^2\underbrace{\int_{s_0}^s\frac{(d-1)  \dot{t}_u^2-1}{\dot{t}_u^2-1} du}_{A_s}+ 2 \s \underbrace{\int_{s_0}^s\frac{\dot{t}_u}{\sqrt{\dot{t}_u^2-1}}dB_u}_{M_s}. \label{loga2}\end{equation}
We will show that both  $M_s$ and $A_s$ converge almost surely when $s$ goes to $\ta$. Fix  $0<\e<1$ and decompose $M_s$ into
$ \langle M \rangle_s = \langle M \rangle_s^+ + \langle M \rangle_s^-$ with
$$\langle M \rangle_s^+:=\int_{s_0}^s\frac{\dot{t}_u^2}{\dot{t}_u^2-1}1_{\{\dot{t}_u^2-1 \geq \e\}}\,du, 
\qquad \langle M \rangle_s^-:=\int_{s_0}^s\frac{\dot{t}_u^2}{\dot{t}_u^2-1}1_{\{\dot{t}_u^2-1 <\e\}}\,du.$$
In the same way, write $ A_s = A_s^+ + A_s^-$ with
$$A_s^+:=\int_{s_0}^s\frac{(d-1)\dot{t}_u^2-1}{\dot{t}_u^2-1}1_{\{\dot{t}_u^2-1 \geq \e\}}\,du, 
\qquad A_s^-:=\int_{s_0}^s\frac{(d-1)\dot{t}_u^2-1}{\dot{t}_u^2-1}1_{\{\dot{t}_u^2-1 <\e\}}\,du.$$
Both $\langle M \rangle_s^+$ and $A_s^+$ are non decreasing and almost surely bounded: 
$$\begin{array}{cll} 
\langle M \rangle_s^+ &   \leq  & \e^{-1} (1+\e) \ta < +\ii, \\
A_s^+ & \leq & \e^{-1} ((d-2)+(d-1)\e) \ta < +\ii, \\
\end{array}
$$
hence they converge almost surely when $s$ goes to $\ta$. Besides, we have 
\begin{equation}\begin{array}{rcl} 
\ds{\int_{s_0}^s\frac{1_{\{\dot{t}_u^2-1 <\e\}}}{\dot{t}_u^2-1}\,du} \leq &  \langle M \rangle_s^- &  \leq \ds{(1+\e) \int_{s_0}^s\frac{1_{\{\dot{t}_u^2-1 <\e\}}}{\dot{t}_u^2-1}\,du},  \\
\\
(d-2)\ds{\int_{s_0}^s\frac{1_{\{\dot{t}_u^2-1 <\e\}}}{\dot{t}_u^2-1}\,du}  \leq & A_s^- & \leq \ds{((d-2)+(d-1)\e) \int_{s_0}^s\frac{1_{\{\dot{t}_u^2-1 <\e\}}}{\dot{t}_u^2-1}\,du},  \\
\end{array}\label{encadrementMA}
\end{equation}
and the processes $\langle M_s\rangle^-$ are $A_s^-$ both convergent or both divergent when $s$ goes to $\ta$. Let us suppose that they are divergent. Necessarily, $\dot{t}_s$ would meet the ball $B(1,\e)$ infinitly often (or it would stay in the ball). Since $\langle M \rangle_s = O(A_s)$ the process $A_s+M_s$ would thus tend to infinity almost surely, so as $\log(\a^2(t_{s_0})/\a^2(t_s))$. The right hand side of Equation (\ref{logtpointcarremoinsun}), and thus the process $\dot{t}_s$ would go to infinity. This contradicts the fact that $\dot{t}_s$ meets the ball $B(1,\e)$ infinitly often. The two processes $M_s$ and $A_s$ are thus almost surely convergent when $s$ goes to $\ta$, and from Equation (\ref{loga2}) $a_s$ converges to a random variable  $a_{\infty} \in (0, +\infty)$. Since $\a(t_s) \to \a(T)=0$ almost surely, necessarily we have $\dot{t}_s \to +\infty$.
\end{proof}

\subsubsection{Preliminaries for the long time asymptotics}
\smallskip
We now turn to the case where $T=+\infty$. 
To highlight the recurrence/transience dichotomy stated in Theorem \ref{theo.causalderivetemp}, let us first begin with the two simplest cases when the Hubble function is a constant, namely the case when $H \equiv 0$  i.e. when the torsion function $\alpha$ is constant, and the case when $H(t)\equiv H> 0$ for all $t>0$. In both cases, the process $\dot{t}_s$ is a one dimensional diffusion process and there exists a real standard Brownian motion $B$ such that, if $H \equiv 0$:
\begin{equation}d \dot{t}_s=  + \frac{d \s^2}{2} \dot{t}_s ds+ \s \sqrt{\dot{t}_s^2-1} dB_s,\label{eqn.stat}\end{equation}
or if $H >0$:
\begin{equation}d \dot{t}_s= -H \times \left( \dot{t}_s^2-1 \right) ds  + \frac{d \s^2}{2} \dot{t}_s ds+ \s \sqrt{\dot{t}_s^2-1} dB_s.\label{eqn.expo}\end{equation}

\begin{lemma} \label{LEM.CASSTAT}
For any starting point $\dot{t}_0 \geq 1$, Equation (\ref{eqn.stat}) admits a unique strong solution $\dot{t}_s$, which is well defined for all $s \in \mathbb R^+$, and such that $\dot{t}_s >1$ almost surely for all $s>0$. Moreover, there exist a real process $u_s$ that converges almost surely when $s$ goes to infinity such that for all $s  \geq 0$:
$$\dot{t}_s = \dot{t}_0 \: \exp \left(\frac{d-1}{2} \s^2 s + \s B_s   + u_s  \right).$$
In particular, the process $\dot{t}_s$ is transient.
\end{lemma}

\begin{proof}
Existence, uniqueness and the lower bound were obtained in Lemma \ref{LEM.UNIEXITEMP}. Applying It\^o's formula to the logarithm function yields
$$ d \log(\dot{t}_s) = \s^2 \left(\frac{d-1}{2} + \frac{1}{2\dot{t}_s^2} \right) ds + \s \sqrt{1 - \dot{t}_s^{-2}} d B_s.$$
Thus, there exists a real standard Brownian motion $B'$ such that 
$$\log(\dot{t}_s)  = \log(\dot{t}_0)  +\frac{d-1}{2} \s^2 s +  \s B_s + \underbrace{\frac{\s^2}{2} \int_0^s \frac{du}{\dot{t}_u^2} - B' \left( \s^2 \int_0^s\frac{du}{\dot{t}_u^4 \left(  1 + \sqrt{1-\dot{t}_u^{-2}}\right)^2} \right)}_{:=u_s},$$
and when $s$ goes to infinity, we have almost surely:
$$\log(\dot{t}_s)  \geq \frac{d-1}{2}  \s^2 s + o(s) > \frac{d-1}{4}  \s^2 s.$$
The two integrals in the definition of $u_s$ are thus convergent, hence the result.
\end{proof}
\medskip

\begin{lemma}  \label{LEM.CASEXPO}
For any starting point $\dot{t}_0 \geq 1$, Equation (\ref{eqn.expo}) admits a unique strong solution $\dot{t}_s$ which is well defined for all $s \in \mathbb R^+$ and satisfies  $\dot{t}_s >1$ for all $s>0$ almost surely. Moreover, the process $\dot{t}_s$ admits an invariant probability measure $ {\n}_{H,\s}$ on $(1, +\ii)$, hence it is ergodic. The measure $\n_{H,\s}$ has the following density with respect to the Lebesgue measure on $(1, +\infty)$: 
$$ \n_{H,\s}(dx):= \frac{1}{Z_{H,\s}} (x^2-1)^{d/2-1} \exp \left( - \frac{2 H}{\sigma^2} x \right)dx.$$
 where $Z_{H,\s}$ is a normalizing constant.
\end{lemma}

\begin{proof}
Again, existence, uniqueness and the lower bound were obtained in Lemma \ref{LEM.UNIEXITEMP}. Finally, one easily checks that the probability measure $\n_{H,\s}$ is invariant, hence the result.
\end{proof}

In the case when $T=+\infty$, the following comparison result will be usefull. 

\begin{lemma} \label{LEM.COMPARPROJ}
Let $(t_s, \dot{t}_s)$ be the solution of Equation (\ref{eqn.ttpoint}) starting from $(t_0, \dot{t}_0) \in  (0,+\infty) \times [1, +\infty)$ where the martingale $M^{\dot{t}}$ is represented by a real Brownian motion $B$: $d M^{\dot{t}}_s =  ( \dot{t}_s^2-1)^{1/2} d B_s$ and consider the two processes $u_s$ and $v_s$ defined as the unique strong solutions starting from $u_0=v_0=\dot{t}_0$ of the following equations: 
$$\begin{array}{l} 
\ds{d u_s=  -H(t_0)\left(u_s^2-1 \right)ds + \frac{d \s^2}{2} u_s ds+ \s \sqrt{u_s^2-1} dB_s}, \\
\ds{d v_s=  -H_{\ii}\left(v_s^2-1 \right)ds+  \frac{d \s^2}{2} v_s ds+ \s \sqrt{v_s^2-1} dB_s}.   
\end{array}
$$
Then, almost surely, for all $0 \leq s <+\ii$ we have $\ds{ u_s \leq \dot{t}_s \leq v_s}$.
\end{lemma}

\begin{proof}
Thanks to the monotonicity of the Hubble function $H$, the Lemma is a direct application of classical comparison resuts. For example, to justify that $u_s \leq \dot{t}_s$, one can apply Theorem 1.1 p.352 of \cite{ikeda}, with $\sigma(s,x)=\sigma \sqrt{x^2-1}$, $x_1(s)=u_s$, $x_2(s)=\dot{t}_s$, $b_1(s,x) = b_2(s,x) = - H(t_0) (x^2-1) + \frac{d \s^2}{2} x$, and finally $\beta_1(s)=b_1(s,x_1(s))$ and $\beta_2(s) = - H(t_s) (x_2(s)^2-1) + \frac{d \s^2}{2} x_2(s)$.
\end{proof}

\subsubsection{Recurrence in exponential case} \label{sec.temprec}

\smallskip
In this paragraph, we establish the fourth point of Theorem \ref{theo.causalderivetemp}, i.e. the recurrence of the process $\dot{t}_s$ in the case $T=+\infty$ and the Hubble function admits a positive limit: $\ds{H_{\infty}= \lim_{t \to +\infty} H(t)>0}$.

\begin{prop}  \label{PRO.ERGO}
Suppose that the Hubble function $H$ is decreasing on $(0,+\infty)$ and that $H_{\infty}>0$. 
Let $(t_s, \dot{t}_s)$ be the solution of Equation (\ref{eqn.ttpoint}) starting from $(t_0, \dot{t}_0) \in  (0,+\infty) \times [1, +\infty)$. Then the non-Markovian process $\dot{t}_s$ is Harris recurrent in $(1, +\ii)$. More precisely, if $f$ is a monotone and $\n_{H_{\ii},\s}-$integrable function, or if it is continuous and bounded, then when $s$ goes to infinity, we have the almost sure convergence: 
$$\frac{1}{s} \: \int_0^s f(\dot{t}_u)du \stackrel{a.s.}{\longrightarrow} \n_{H_{\ii},\s}(f):=\int f d\n_{H_{\ii},\s} . $$
\end{prop}

\begin{proof}
Let $(t_s, \dot{t}_s)$ be the solution of Equation (\ref{eqn.temp}) starting from $(t_0, \dot{t}_0)$.
Let $z_s$ and $z_s^n$, $n \in \mathbb N$, be the processes defined as follows. 
The process $z_s$  is the strong solution starting from $z_0=\dot{t}_0$ of the equation
$$d z_s= -H_{\ii} \times \left( |z_s|^2-1 \right)ds  + \frac{d \s^2}{2} z_s ds+ \s \sqrt{|z_s|^2-1} \,dB_s.$$
For all $n \in \mathbb N$, the processes $z^n_s$ coincide with $\dot{t}_s$ on the interval $[0,n]$ and satisfy the following equations on $[n, +\ii[$
$$d z^n_s= -H(t_0+n) \times \left( |z^n_s|^2-1 \right)ds  + \frac{d \s^2}{2} z^n_s ds+ \s \sqrt{|z^n_s|^2-1} \,dB_s.$$
Almost surely, for all $n \geq 0$  and for all $s \geq 0$, we then have $z^n_s \leq \dot{t}_s \leq z_s$. 
Indeed, the inequality $\dot{t}_s \leq z_s$ was obtained in Lemma \ref{LEM.COMPARPROJ} and the other inequality $z^n_s \leq \dot{t}_s$ is also a consequence of Lemma \ref{LEM.COMPARPROJ}, taking initial conditions $t_0' = t_n \geq t_0+n$ and 
$\dot{t}_0'= \dot{t}_n$. From Lemma \ref{LEM.CASEXPO}, the two processes $z^0_s$ and $z_s$ are ergodic in $(1, +\ii)$, hence they are Harris recurrent and so is $ \dot{t}_s$.
Take an increasing and $\n_{H_{\ii},\s}-$integrable function $f$ and fix $\e>0$.
The function $f$ is integrable against $\n_{H(t_0+n),\s}$ for all $n \in \mathbb N$ and when $n$ goes to infinity, we have 
$\n_{H(t_0+n),\s}(f) \longrightarrow \n_{H_{\ii},\s}(f)$. Take $n$ large enough so that
$$ | \n_{H(t_0+n),\s}(f) - \n_{H_{\ii},\s}(f)| \leq \e.$$
Since $z^n_s \leq \dot{t}_s \leq z_s$ for all $s \geq 0$, we have almost surely:
$$\ds{\int_0^s f(z^n_u)du \leq \int_0^s f(\dot{t}_u)du \leq \int_0^s f(z_u)du}.$$
The integer $n$ being fixed, from the ergodic Theorem, we have almost surely, when  $s$ goes to infinity:
$$\n_{H(t_0+n),\s}(f) \leq \liminf_{s \to + \ii} \frac{1}{s} \: \int_0^s f(\dot{t}_u)du \leq \limsup_{s \to + \ii}\frac{1}{s} \; \int_0^s f(\dot{t}_u)du \leq \n_{H_{\ii},\s}(f),$$
hence
$$\n_{H_{\ii},\s}(f)- \e  \leq \liminf_{s \to + \ii} \frac{1}{s} \: \int_0^s f(\dot{t}_u)du \leq \limsup_{s \to + \ii}\frac{1}{s} \; \int_0^s f(\dot{t}_u)du \leq \n_{H_{\ii},\s}(f).$$
Letting $\varepsilon$ go to zero, we get the announced result. As any regular function can be written as the difference of two monotone functions, the convergence extends $f \in C^{1}_b=\{ f, \; f' \; \textrm{is bounded} \; (1, +\ii)\}$, and finally by regularization, to $f \in C^0_b((1,+\ii), \mathbb R)$.
\end{proof}

\subsubsection{Almost sure transience if the growth is at most polynomial} \label{SEC.ASYMPTEMPTRANS}
\smallskip
We now deal with the second point of Theorem \ref{theo.causalderivetemp}, i.e. the transience of the process $\dot{t}_s$ in the case $T=+\infty$ and the growth rate of the expansion function $\alpha$ is at most polynomial. 
Let us first prove the following lemma which is valid as soon as $H_{\ii}=0$ i.e. in both polynomial and subexponential cases.
\begin{lemma}  \label{lem.control.exp}
Let $(t_s, \dot{t}_s)$ be the solution of Equation (\ref{eqn.ttpoint}) starting from $(t_0, \dot{t}_0) \in (0,+\infty) \times [1, +\infty)$. If $H_{\ii}=0$, then almost surely, when $s$ goes to infinity we have
\[ 
\log \left( \a(t_s)\dot{t}_s\right) = \log \left( \int_{.}^{t_s} \a(u) du \right)= \frac{d-1}{2} \s^2 \times s + o(s). 
\]
\end{lemma}

\begin{proof}
From Equation (\ref{eqn.temp}), It\^o's formula gives $\log \left( \a(t_s)\dot{t}_s\right) = \frac{d-1}{2} \s^2  \: s +  v_s$, with
\begin{equation} v_s:= \log( \a(t_0)\dot{t}_0 )+\int_0^s \frac{H(t_u)}{\dot{t}_u}du + \int_0^s \frac{\s^2}{2 \dot{t}_u^2}du +\s  B_s -
\s \int_0^s \frac{\dot{t}_u^{-2}}{1+\sqrt{1-\frac{1}{\dot{t}_u^2}}} \: dB_u. \label{eqn.v}\end{equation}
From the law of iterated logarithm, almost surely when $s$ goes to infinity,  we have
\begin{equation} |B_s| + \left|\int_0^s \frac{\dot{t}_u^{-2}}{1+\sqrt{1-\frac{1}{\dot{t}_u^2}}} \: dB_u\right| = o(s).\label{eqn.petito1}\end{equation}
Otherwise,  almost surely when $s$ goes to infinity, we also have
\begin{equation}\int_0^s \frac{H(t_u)}{\dot{t}_u} du + \int_0^s \frac{du}{ \dot{t}_u^2} = o (s).\label{eqn.petito2}\end{equation}
Indeed, since $H_{\ii}=0$, $\dot{t}_s \geq 1$ and $t_s \geq s$ for all $s \geq 0$, when $s$ goes to infinity, we have naturally:
$$\int_0^s \frac{H(t_u)}{\dot{t}_u} du=o(s).$$
Now, fix $\eta>0$ and consider the (deterministic) stopping time $\ta_{\eta}:=\inf \{s>0, \; H(s) \leq \eta\}$.
Let $z_s$ be the diffusion process that coincides with $\dot{t}_s$ on  $[0, \ta_{\eta}]$ and which is solution to 
the following stochastic differential equation on  $[\ta_{\eta},+\ii)$:
\[
d z_s= - \eta \times \left( z_s^2-1 \right)ds  + \frac{d \s^2}{2} z_s ds+ \s \sqrt{z_s^2-1} \,dB_s.
\]
From Lemma \ref{LEM.COMPARPROJ} (with initial conditions $t_0'=t_{\ta_{\eta}}$ and $\dot{t}_0'=\dot{t}_{\ta_{\eta}}$), almost surely, one has
\[ 
z_s \leq \dot{t}_s,\;\; \textrm{for all $s \geq 0$}. 
\]
From Lemma \ref{LEM.CASEXPO}, the process $z_s$ is ergodic in  $(1,+\ii)$, with invariant probability $\nu_{\eta,\s}$. The function $x \mapsto 1/x^2$ being integrable against  $\nu_{\eta,\s}$, from the ergodic Theorem, we have almost surely when $s$ goes to infinity:
\[
\frac{1}{s} \: \int_0^s \frac{du}{ z_u^2} \longrightarrow C(\eta,\s^2):=\int_{1}^{+\ii} x^{-2} \n_{\eta, \s}(x)dx
= \frac{\displaystyle{\int_{1}^{+\ii} x^{-2}  (x^2-1)^{d/2-1} e^{ - \frac{2 \eta}{\sigma^2} x }dx}}{\displaystyle{\int_1^{+\infty}(x^2-1)^{d/2-1}  e^{- \frac{2 \eta}{\sigma^2} x }dx}}.
\]
Settingt $\eta'=2 \eta/\s^2$ and performing the change of variable $u=\eta' (x-1)$, for $\eta'$ small enough:
{\small $$C(\eta,\s^2)=
\frac{\ds{\int_{0}^{+\ii} \frac{\eta'^2}{(u+\eta')^2}  (u(u+2\eta'))^{d/2-1} e^{-u} du}}{\ds{\int_{0}^{+\ii}(u(u+2\eta'))^{d/2-1} e^{-u} du} }\leq
\frac{\ds{\int_{0}^{+\ii} \frac{\eta'^2}{(u+\eta')^2}  (u(u+1))^{d/2-1} e^{-u} du}}{\ds{\int_{0}^{+\ii}u^{d-2} e^{-u} du} }.$$}\par
\noindent
The parameter $\s$ being fixed, from the dominated convergence theorem, $C(\eta,\s^2)$ goes to zero with $\eta$. Let $\epsilon>0$ and $\eta$ small enough so that $ C(\eta,\s^2)\leq \epsilon/2$. Almost surely, for $s$ large enough, we get
$$\frac{1}{s} \; \int_0^s \frac{du}{ \dot{t}_u^2} \leq \frac{1}{s} \: \int_0^s \frac{du}{ z_u^2}  \leq 2 C(\eta,\s^2)\leq \epsilon,$$
hence the estimate (\ref{eqn.petito2}). The two estimates (\ref{eqn.petito1}) and (\ref{eqn.petito2}) show that $v_s=o(s)$, or in other words
$$ \log \left(\a(t_s)\dot{t}_s\right) = \frac{d-1}{2}\s^2 \times s + o(s), \;\; \hbox{and by integration} \;\; \log \left( \int_{.}^{t_s} \a(u) du \right)= \frac{d-1}{2}\s^2 \times s + o(s).$$
\end{proof}

From Lemma \ref{lem.control.exp}, we can now deduce the transience of the temporal process $(t_s, \dot{t}_s)$ when $T=+\infty$ and the expansion is at most polynomial. 

\begin{cor}\label{cor.trans.poly}
If $T=\infty$ and the growth of the torsion function is at most polynomial, then the process $\dot{t}_s$ is almost surely transient.
\end{cor}

\begin{proof} 
From Lemma \ref{lem.control.exp}, if the expansion is at most polynomial, we have almost surely when $s$ goes infinity
\[
 \frac{\log \left( \a(t_s)\dot{t}_s\right)}{ \log \left( \int_{.}^{t_s} \a(u) du \right)}= 1+o(1) \quad \hbox{and} \quad
\limsup_{s \to +\infty} \frac{\log \left( \a(t_s)\right)}{ \log \left( \int_{.}^{t_s} \a(u) du \right)}<1.
\]
We deduce that almost surely
\[
\liminf_{s \to +\infty}\frac{\log \left(\dot{t}_s\right)}{ \log \left( \int_{.}^{t_s} \a(u) du \right)}>0,
\]
 hence the result.
\end{proof}

Moreover, we can give explicit speeds of divergence. 
\begin{prop} \label{PRO.ASYMPC}Let $(t_s, \dot{t}_s)$ be the solution of Equation (\ref{eqn.ttpoint}) starting from $(t_0, \dot{t}_0) \in (0,+\infty) \times [1, +\infty)$.
Suppose that $H_{\infty}=0$ and that the torsion function $\alpha$ has polynomial growth of rate $c \in [0, +\infty)$ at infinity in the sense that $H(t) \times t$ converges to $c>0$ when $t$ goes to infinity.
Then almost surely, when $s$ goes to infinity, the process $\dot{t}_s$ is transient and we have
$$ \frac{1}{s} \log(\dot{t}_s) \longrightarrow \frac{d-1}{2} \frac{\s^2}{1+c},\quad \frac{1}{s} \log(\a(t_s)) \longrightarrow \frac{d-1}{2}  \frac{\s^2 \: c}{1+c}.$$
In particular, recalling that $a_s=\alpha(t_s) \sqrt{\dot{t}_s^2-1}$, we have almost surely, when $s$ goes to infinity:
$$ \frac{1}{s} \log\left(\frac{a_s}{\a^2(t_s)}\right)  \longrightarrow \frac{d-1}{2} \s^2 \:\left(\frac{1-c}{1+c}\right).$$
\end{prop} 
\noindent

\begin{proof}Let us suppose that, when $t$ goes to infinity, $H(t) \times t$ tends to $c \in [0, +\ii)$. Let $0< \e<1$, and $t_0$ large enough so that
for all $t \geq t_0$: $c-\e \leq H(t) \times t \leq c+\e$. There exists two constants $c_{0}$ and $c_0'$ such that, for all $t \geq t_0$: 
$$  (c-\e+1) \log(t) + c_0 \leq \log \left( \int_{t_0}^{t} \a(u) du \right) \leq (c+\e+1) \log(t) + c_0'. $$
From Lemma \ref{lem.control.exp}, almost surely when $s$ goes to infinity
$$ \lim_{s \to +\ii} \frac{1}{s} \log \left( \int_{t_0}^{t_s} \a(u) du \right) =\frac{d-1}{2}  \s^2.$$
Thus, almost surely when $s$ goes to infinity
$$  \frac{d-1}{2} \frac{\s^2}{1+c+\e} \leq \liminf_{s \to +\ii} \frac{1}{s} \log(t_s)  \leq \limsup_{s \to +\ii} \frac{1}{s} \log(t_s)\leq \frac{d-1}{2}  \frac{\s^2}{1+c-\e},$$
and letting $\varepsilon$ go to zero:
$$\lim_{s \to +\ii} \frac{1}{s} \log(t_s)  =\frac{d-1}{2} \frac{\s^2}{1+c}.$$ 
Moreover, since $t\mapsto \a(t)$ grows as $t^c$ at infinity, we have  
$$ \lim_{s \to +\ii} \frac{1}{s} \log(\a(t_s)) = \frac{d-1}{2} \frac{\s^2 \: c}{1+c}, \;\; \hbox{and from Lemma \ref{lem.control.exp} again} \;\;  \lim_{s \to +\ii} \frac{1}{s} \log(\dot{t}_s) =\frac{d-1}{2} \frac{\s^2}{1+c}. $$
In particular
$$ \lim_{s \to +\ii} \frac{1}{s} \log\left(\frac{a_s}{\a^2(t_s)}\right) = \lim_{s \to +\ii} \frac{1}{s} \log\left(\frac{\dot{t}_s}{\a(t_s)}\right) = \frac{d-1}{2}  \s^2 \:\left(\frac{1-c}{1+c}\right).$$ 
\end{proof}

\subsubsection{Transience in probability in the subexponential case} 
\smallskip
The last case to consider is the one where the torsion function $\alpha$ has a subexponential growth. Let us first prove that in that case, the temporal derivative $\dot{t}_s$ goes to infinity in probability. The next proposition shows that it is the case as soon as $H_{\infty}=0$.

\begin{prop} \label{PRO.TRANSPROB}Let $(t_s, \dot{t}_s)$ be the solution of Equation (\ref{eqn.ttpoint}) starting from $(t_0, \dot{t}_0) \in (0,+\infty) \times [1, +\infty)$. Suppose that $H_{\infty}=0$, then for all $R>1$ we have $\liminf_{s \to +\infty} \mathbb P(\dot{t}_s> R) =1$.
\end{prop}

\begin{proof}
Let us proceed as in the proof of Lemma \ref{lem.control.exp}. Namely, fix $R>1$ and $\eta>0$ and consider the (deterministic) stopping time $\ta_{\eta}:=\inf \{s>0, \; H(s) \leq \eta\}$.
Let $z_s$ be the diffusion process that coincides with $\dot{t}_s$ on  $[0, \ta_{\eta}]$ and which is solution to 
the following stochastic differential equation on  $[\ta_{\eta},+\ii)$:
\[
d z_s= - \eta \times \left( z_s^2-1 \right)ds  + \frac{d \s^2}{2} z_s ds+ \s \sqrt{z_s^2-1} \,dB_s.
\]
From Lemma \ref{LEM.COMPARPROJ} (with initial conditions $t_0'=t_{\ta_{\eta}}$ and $\dot{t}_0'=\dot{t}_{\ta_{\eta}}$), almost surely, one has
$z_s \leq \dot{t}_s$ for all $s \geq 0$, so that $\mathbb P( \dot{t}_s > R) \geq \mathbb P( z_s > R)$. Moreover, by Lemma \ref{LEM.CASEXPO} the process $z_s$ is ergodic with invariant measure $\nu_{\eta, \sigma}$ and $\lim_{s \to +\infty} \mathbb P( z_s > R) = \nu_{\eta, \sigma}([R, +\infty))$ where 
\[
\nu_{\eta, \sigma}([R, +\infty))=\frac{\displaystyle{\int_{R}^{+\ii}  (x^2-1)^{d/2-1} e^{ - \frac{2 \eta}{\sigma^2} x }dx}}{\displaystyle{\int_1^{+\infty}(x^2-1)^{d/2-1}  e^{- \frac{2 \eta}{\sigma^2} x }dx}} = \frac{\displaystyle{\int_{\eta(R-1)}^{+\ii}  (u(u+2\eta))^{d/2-1} e^{ - \frac{2}{\sigma^2} u }du}}{\displaystyle{\int_0^{+\infty}  (u(u+2\eta))^{d/2-1} e^{ - \frac{2}{\sigma^2} u }du}}
\]
goes to one when $\eta$ goes to zero, hence the result.
\end{proof}

\subsubsection{Necessarily and sufficient condition for the almost sure transience}
\smallskip
We now give a necessarily and sufficient criterion that ensures the almost sure transience 
of $\dot{t}_s$ in the case $H_{\infty}=0$. Let us consider the two functions
\[
f(x):=-\int_x^{+\infty} \frac{du}{(u^2-1)^{d/2}}, \qquad g(x):=-\int_x^{+\infty} \frac{du}{u^{d-2}(u^2-1)}.
\]
We have 
\[
f'(x):=\frac{1}{(x^2-1)^{d/2}}, \qquad g'(x):=\frac{1}{x^{d-2}(x^2-1)}.
\]
and
\[
f''(x):=\frac{- d \, x }{(x^2-1)^{d/2+1}}, \qquad g''(x):=-\frac{d }{x^{d-3}(x^2-1)^2}+\frac{(d-2)}{x^{d-1}(x^2-1)^2}.
\]
From It\^o's formula, we then have
\begin{eqnarray}
f(\dot{t}_s)  = f(\dot{t}_0) -\int_0^s \frac{H(t_u)du}{(\dot{t}_u^2-1)^{d/2-1}}  + M_s, \label{eq.trans1} \\
\nonumber \\
g(\dot{t}_s)   = g(\dot{t}_0) -\int_0^s \frac{H(t_u)}{\dot{t}_u^{d-2}} du + R_s, \label{eq.trans2}
\end{eqnarray}
where
\[
M_s  :=\sigma \int_0^s  \frac{1}{(\dot{t}_u^2-1)^{\frac{d-1}{2}}} d B_u, \quad 
R_s :=\frac{d-2}{2} \sigma^2\int_0^s  \frac{du}{\dot{t}_u^{d-1}(\dot{t}_u^2-1)} +\sigma \int_0^s  \frac{1}{\dot{t}_u^{d-2} \sqrt{\dot{t}_u^2-1}} d B_u.
\]

Here is a first necessarily and sufficient criterion:

\begin{prop}\label{prop.ssi}
The process $\dot{t}_s$ goes almost surely to infinity with $s$ if and only if 
\[
\int_0^{+\infty} \frac{H(t_s)}{\dot{t}_s^{d-2}}ds<+\infty \quad \hbox{almost surely}.
\]
\end{prop}

\begin{proof}
To simplify the expressions, define 
\[
I_s :=\int_0^s \frac{H(t_u)du}{(\dot{t}_u^2-1)^{d/2-1}} , \qquad J_s:=\int_0^s \frac{H(t_u)}{\dot{t}_u^{d-2}} du.
\]
Note that if $\dot{t}_s$ is transient, then both integrals $I_s$ and $J_s$ are of the same nature i.e. converge or diverge simultaneously.
If $\dot{t}_s$ goes almost surely to infinity with $s$, then $f(\dot{t}_s)$ goes to zero almost surely . From Equation (\ref{eq.trans1}), we deduce that the local martingale $M_s$ has the same asymptotic behavior as the non-decreasing integral $I_s$. This is possible only if both quantities are convergent, in other words  the integral $I_s$ (and thus $J_s$) converge almost surely when $s$ goes to infinity.\\ 
Now, if  $J_s$ converges almost surely when $s$ goes to infinity, then from Equation (\ref{eq.trans2}), the process $R_s$ is bounded above and consequently, it converges almost surely (the martingale term is dominated by the term of finite variation). Finally, $g(\dot{t}_s)$ converges almost surely, and since $R_s$ converges, this is possible only if $\dot{t}_s$ goes to infinity, hence the result.
\end{proof}

To prove Proposition \ref{pro.H3ssi}, we are left to translate the convergence of $J_s$ in terms of the rate of decrease of the Hubble function $H$, or equivalently in terms of integrability of $H^d$.


\begin{proof}[Proof of Proposition \ref{pro.H3ssi}]
From Proposition \ref{prop.ssi},  $\dot{t}_s$ goes almost surely to infinity with $s$ if and only if the integral $J_s$ converges almost surely when $s$ goes to infinity. Integrating by parts, we get
\[
J_s=\int_{0}^s \frac{H(t_u)}{\dot{t}_u^{d-2}}du = \int_{0}^s \frac{du}{\left[H(t_u)\dot{t}_u\right]^{d-2}} \times H^{d-1}(t_s) - (d-1) \int_0^s \left(   \int_{0}^v \frac{du}{\left[H(t_u)\dot{t}_u\right]^{d-2}}\right) H^{d-2}(t_v) H'(t_v) \dot{t}_v dv.
\]
Otherwise, we have the following comparison result:
\begin{lemma}\label{lem.encadre}
There exists two deterministic constants $0<\kappa < K <\infty$ such that, almost surely when $s$ goes to infinity
\[
\kappa \leq \liminf_{s \to +\infty} \frac{1}{s} \, \int_{0}^s \frac{du}{\left[H(t_u)\dot{t}_u\right]^{d-2}} \leq \limsup_{s \to +\infty} \frac{1}{s} \, \int_{0}^s \frac{du}{\left[H(t_u)\dot{t}_u\right]^{d-2}} \leq K.
\]
\end{lemma}
\noindent
Let us admit Lemma \ref{lem.encadre} for a moment. Recall that if $\alpha$ has subexponential growth, we have 
\[
\lim_{t \to +\infty} \frac{\log \left( \alpha(t)\right)}{\log \left(\int^t \alpha(u) du\right)}=1,
\]
and recall also that from Lemma \ref{lem.control.exp}, we have almost surely when $s$ goes to infinity
\[
\log \left( \int_{.}^{t_s} \a(u) du \right)= \frac{d-1}{2} \s^2 \times s + o(s). 
\]
Thus, almost surely we have
\[
\frac{2 \kappa}{(d-1)\sigma^2} \leq \liminf_{s \to +\infty} \frac{1}{\log(\alpha(t_s))} \, \int_{0}^s \frac{du}{\left[H(t_u)\dot{t}_u\right]^{d-2}} \leq \limsup_{s \to +\infty} \frac{1}{\log(\alpha(t_s))} \, \int_{0}^s \frac{du}{\left[H(t_u)\dot{t}_u\right]^{d-2}} \leq \frac{2 K}{(d-1)\sigma^2},
\]
and we deduce the two following asymptotic bounds (recall that $-H' \geq 0$ since $H$ is nonincreasing) 
\begin{align*}
J_s & \leq \frac{4 K}{(d-1)\sigma^2} \left[  \log(\alpha(t_s)) \times H^{d-1}(t_s) - (d-1) \int_0^s  \log(\alpha(t_v) ) H^{d-2}(t_v) H'(t_v) \dot{t}_v dv\right],\\
\\
J_s & \geq \frac{ \kappa}{(d-1)\sigma^2}   \left[  \log(\alpha(t_s)) \times H^{d-1}(t_s) - (d-1) \int_0^s  \log(\alpha(t_v) ) H(t_v)^{d-2} H'(t_v) \dot{t}_v dv\right].
\end{align*}
A new integration by parts shows that  
\[
\int_0^t H^d(s)ds = H^{d-1}(t) \log(\alpha(t) )  - (d-1) \int_0^t  \log(\alpha(s) ) H^{d-2}(s) H'(s) ds.
\]
In other words, almost surely for $s$ sufficiently large, we have
\[
\frac{ \kappa}{(d-1)\sigma^2} \int_0^{t_s} H^d(u)du  \leq J_s \leq \frac{4 K}{(d-1)\sigma^2} \int_0^{t_s} H^d(u)du,
\]
hence the result.
\end{proof}

\begin{proof}[Proof of Lemma \ref{lem.encadre}]
Let us recall that there exists a Brownian motion $B$ such that 
\[
d \dot{t}_s = -H(t_s) \left( \dot{t}_s^2-1\right) ds  + \frac{d \sigma^2}{2} \dot{t}_s ds +\sigma \sqrt{\dot{t}_s^2-1}d B_s,
\]
and 
\[
d \sqrt{\dot{t}_s^2-1} = -H(t_s) \dot{t}_s \sqrt{ \dot{t}_s^2-1}ds  + \frac{d \sigma^2}{2} \sqrt{\dot{t}_s^2-1}ds +\frac{d-1}{2}\frac{\sigma^2}{\sqrt{\dot{t}_s^2-1}}ds+ \sigma \dot{t}_s d B_s.
\]
Straightforward It\^o calculus shows that the process $v_s:= H(t_s) (\dot{t}_s +  \sqrt{\dot{t}_s^2-1})$ is then solution to the following stochastic differential equation
\begin{equation}
d v_s =  - \frac{v^2_s}{2}  \left( 1-\frac{H'(t_s)}{H^2(t_s)}\right) ds +  \frac{H^2(t_s)}{2}  \left( 1 +\frac{H'(t_s)}{H^2(t_s)}\right)ds + \frac{d \sigma^2}{2}  v_s ds + (d-1)\sigma^2\frac{  H^2(t_s) v_s}{v^2_s-H^2(t_s)}ds  + \sigma v_s d B_s \\
\end{equation}
\noindent
Recall that in the subexponential case, $H(t)$ goes to zero when $t$ goes to infinity and $-H'(t)/H^2(t)$ is bounded by Hypothesis \ref{hypo.2}.
So let $0< \varepsilon<<1$, and $s_0$ large enough (deterministic) so that for all $s\geq s_0$ we have $H(s) \leq \varepsilon$ and $\kappa>0$ such that $\limsup_{s \to +\infty} - \frac{H'(s)}{H^2(s)}  \leq \kappa$. Without loss of generality, we can suppose that $\kappa \geq 1$.
As $t_s >s$ almost surely, for $s >s_0$, we have $H(t_s) \leq \varepsilon$ almost surely, in particular  $\eta:=H(t_{s_0}) \leq \varepsilon$. By the classical comparison results, we get that almost surely, for $s>s_0$ we have
\begin{equation}\label{eqn.encadrement}
y_s \leq v_s \leq x_s,
\end{equation}
where $x_s$ and $y_s$ are the solutions starting from $v_{s_0}> \eta$ of the following stochastic differential equations
\begin{align*}
d x_s & =- \frac{x^2_s}{2}   ds + \frac{\varepsilon^2}{2}  ds + \frac{d \sigma^2}{2}  x_s ds + \frac{(d-1) \sigma^2 \varepsilon^2 }{x_s-\eta}ds  + \sigma x_s d B_s,
\\
d y_s & =  - \kappa y^2_s ds + \frac{d \sigma^2}{2}  y_s ds   + \sigma y_s d B_s.
\end{align*}
Applying It\^o's formula to the logarithm function, we get 
\begin{align*}
(x_s-\eta) & =(x_{s_0}-\eta) \exp \left[    - \frac{1}{2}\int_{s_0}^s (x_u-\eta)   du + \left(  \frac{d \sigma^2}{2}   -\eta \right) (s-s_0) +  \frac{d \sigma^2 \eta + \varepsilon^2 -\eta^2}{2}  \int_{s_0}^s \frac{du}{x_u-\eta}     \right] \\
& \times \exp \left[   \sigma (B_s -B_{s_0}) + \sigma^2  \int_{s_0}^s \frac{(d-1)\varepsilon^2 - x_u^2/2}{(x_u-\eta)^2}du   + \sigma \eta \int_{s_0}^s  \frac{dB_u}{x_u-\eta}\right], 
\end{align*}
from which we deduce that $x_s$ is well defined for $s>s_0$ and satisfies $x_s > \eta$ almost surely. In the same way, we have 
\[
y_s=y_0 \exp \left[    -  \kappa \int_{s_0}^s y_u  du + \frac{d -1}{2} \sigma^2 (s-s_0) + \sigma (B_s -B_{s_0}) \right],
\]
from which we deduce that  $y_s$ is well defined for $s>s_0$ and satisfies $y_s > 0$ almost surely. Moreover, both processes $x_s$ and $y_s$ are ergodic with invariant probability measures $\mu$ and $\nu$ such that  
\[
\mu (dx) =\frac{1}{Z_{\mu}}\frac{(x-\eta)^{2(d-1)\varepsilon^2/\eta^2}}{x^{2(d-1)\varepsilon^2/\eta^2-(d-2)}} \exp \left(  - \frac{1}{\sigma^2} x -\frac{\varepsilon^2}{\sigma^2 x}  \left( 1-\frac{2(d-1)\sigma^2}{\eta}\right)  \right) \mathds{1}_{x>\eta} dx,
\]
and
\[
 \quad \nu(dy) =\frac{1}{Z_{\nu}} y^{d-2} \exp \left(  - \frac{2\kappa }{\sigma^2} y \right) \mathds{1}_{y>0}dy,
\]
where $Z_{\mu}$ and $Z_{\nu}$ are normalizing constants. The function $x \mapsto x^{2-d}$ being integrable against both $\mu$ and $\nu$, applying the ergodic Theorem, we get that almost surely when $s$ goes to infinity:
\[
\frac{1}{s} \int_{s_0}^s \frac{du}{x_u^{d-2}} \to \int \frac{\mu(dx)}{x^{d-2}} \in (0, +\infty), \qquad \frac{1}{s} \int_{s_0}^s \frac{du}{y_u^{d-2}} \to \int \frac{\nu(dy)}{y^{d-2}} \in (0, +\infty).
\]
From (\ref{eqn.encadrement}), we then deduce that almost surely when $s$ goes to infinity:
\[
\int \frac{\mu(dx)}{x^{d-2}} \leq \liminf_{s \to +\infty} \frac{1}{s} \int_{s_0}^s \frac{du}{v_u^{d-2}} \leq  \limsup_{s \to +\infty} \frac{1}{s} \int_{s_0}^s \frac{du}{v_u^{d-2}} \leq \int \frac{\nu(dy)}{y^{d-2}}.
\]
Otherwise, we have naturally $H(t_s) \dot{t}_s \leq v_s \leq 2 H(t_s) \dot{t}_s$, so that
\[
\frac{1}{2s} \int_{s_0}^s \frac{du}{\left[H(t_u) \dot{t}_u\right]^{d-2}} \leq \frac{1}{s} \int_{s_0}^s \frac{du}{v_u^{d-2}} \leq \frac{1}{s} \int_{s_0}^s \frac{du}{\left[H(t_u) \dot{t}_u\right]^{d-2}},
\]
and finally
\[
\int \frac{\mu(dx)}{x^{d-2}} \leq \liminf_{s \to +\infty}  \frac{1}{s} \int_{s_0}^s \frac{du}{\left[H(t_u) \dot{t}_u\right]^{d-2}} \leq  \limsup_{s \to +\infty} \frac{1}{s} \int_{s_0}^s \frac{du}{\left[H(t_u) \dot{t}_u\right]^{d-2}} \leq 2  \int \frac{\nu(dy)}{y^{d-2}}.
\]
\end{proof}

\subsubsection{A convergence criterion}\label{sec.clock}
\smallskip
We conclude this Section by stating a convergence result concerning an additive functional of the temporal process.
Fix $s_0>0$, we are interested in the convergence/divergence of the following integral as $s$ goes to the explosion time $\tau$:
$$C_s:=\sigma^2 \int_{s_0}^s \frac{\a^2(t_u)}{a_u^2} du = \s^2 \int_{s_0}^s \frac{du}{\dot{t}^2_u-1}.$$
As noticed in the beginning of Sect. \ref{sec.proofs}, the spatial process $(x_s, \dot{x}_s)$ can be seen as an inhomogeneous diffusion on $T^1 M$, parametrized by the clock $C_s$. Therefore, the convergence of this time change is of primer importance to understand the asymptotic behavior spatial components of the relativistic diffusion.
\begin{cor} \label{cor.clock}Let $(t_s, \dot{t}_s)$ be the solution of Equation (\ref{eqn.ttpoint}) starting from $(t_0, \dot{t}_0) \in (0,+\infty) \times [1, +\infty)$. Then, when $s$ goes to $\tau$, we have the following asymptotic behaviors:
\begin{enumerate}
\item if $T<+\infty$, the process $C_s$ is almost surely convergent;
\item if $T=+\ii$ and if the growth of $\alpha$ is at most polynomial, then $C_s$ is almost surely convergent;
\item if $T=+\ii$ and if the growth of $\alpha$ is at subexponential with $H^3 \in \mathbb L^1$ in the case $d>3$ and $H^3 \in \mathbb L^{1^{-}}$ in the case $d=3$, then $C_s$ is almost surely convergent;
\item if $T=+\ii$ and $\alpha$ is of exponential growth or is of subexponential growth with $H^3 \notin \mathbb L^1$, $C_s$ goes to infinity with $s$ almost surely. 
\end{enumerate}
\end{cor}

\begin{proof}
If $T<+\infty$, Proposition \ref{PRO.CONVA} ensures that the explosion time $\ta$ is finite almost surely and that $\dot{t}_s$ goes to infinity when $s$ goes to $\tau$. We thus have $ \lim_{s \to \ta}\uparrow  C_s < +\infty$ almost surely, hence the first point.
Now, if $T=+\ii$ and if the growth of $\alpha$ is at most polynomial, we know by Proposition \ref{PRO.ASYMPC} that $\dot{t}_s$ goes exponentially fast to infinity with $s$, hence the second point. If the expansion is exponential i.e. if $H_{\infty}>0$, the almost-sure transience of $C_s$ is an immediate consequence of Proposition \ref{PRO.ERGO}, hence the point 4. 
Let us now concentrate on the only remaining case i.e. the subexponential case. An integration by parts gives
\begin{equation}\label{eqn.unsurxdeux}
D_s:=\int_{0}^s \frac{du}{\dot{t}_u^2}  = \int_{0}^s \frac{H^2(t_u)du}{H^2(t_u)\dot{t}_u^2} 
 =  \int_{0}^s \frac{du}{H^2(t_u)\dot{t}_u^2} \times H^2(t_s) - 2 \int_0^s \left(\int_{0}^u \frac{du}{H^2(t_v)\dot{t}_v^2} \right)  H(t_u)H'(t_u) \dot{t}_u du.
\end{equation}
\underline{Suppose first that $d>3$}. With the same notations as in the proof of Lemma \ref{lem.encadre} and following the same reasonning, since the function $x \mapsto 1/x^2$ is integrable against both $\mu$ and $\nu$, we get that there exists two constants $0<\gamma<\Gamma<+\infty$ such that 
\[
\gamma \leq \liminf_{s \to +\infty} \frac{1}{\log(\alpha(t_s))} \, \int_{0}^s \frac{du}{H^2(t_u)\dot{t}_u^2} \leq \limsup_{s \to +\infty} \frac{1}{\log(\alpha(t_s))} \, \int_{0}^s \frac{du}{H^2(t_u)\dot{t}_u^2} \leq \Gamma,
\]
from which we deduce as in the end of the proof of Lemma \ref{lem.encadre}  that 
\begin{equation}\label{eqn.encadretdeux}
\frac{\gamma}{2} \int_{t_0}^{t_s}H^3(u)du \leq D_s \leq \frac{\Gamma}{2} \int_{t_0}^{t_s}H^3(u)du.
\end{equation}
If $H^3 \in \mathbb L^1$, we know from Proposition \ref{pro.H3ssi} that $\dot{t}_s$ is transient almost surely, so that the asymptotic behavior of $C_s$ is similar to the one of $D_s$ and thanks to (\ref{eqn.encadretdeux}) it is almost surely convergent. On the contrary, if $H^3 \notin \mathbb L^1$, since $C_s$ is bounded below by $D_s$, the comparison (\ref{eqn.encadretdeux}) shows that it goes almost surely to infinity with $s$, hence the result if $d>3$. \par
\smallskip
\noindent
\underline{Suppose now that $d=3$}. The above reasonning does not apply because the function $y \mapsto 1/y^2$ is not integrable anymore against $\nu$. Nevertheless, this function is still integrable against $\mu$ and the lower bound in (\ref{eqn.encadretdeux}) still holds true so that we deduce that $C_s$ goes almost surely to infinity with $s$ if $H^3 \notin \mathbb L^1$. 
Now consider a Hubble function $H$ such that $H^3 \in \mathbb L^1{^-}$ i.e. there exists $\eta>0$ such that $H \in \mathbb L^{3-\eta}$. Without loss of generality, we can suppose that $\eta<1$. From the following integration by parts, 
\[
\int_0^t H^{3-\eta}(u)du = \int_0^t H^{1-\eta}(u)du  \times H^2(t) - 2\int_0^t \left( \int_0^s H^{1-\eta}(s)ds \right) H'(s)H(s)ds, 
\]
recalling that $H'\leq 0$ and introducing $\rho:=\int_{\mathbb R^+} H^{3-\eta}(u)du \in (0,+\infty)$,  we deduce that for all $t >0$, we have $\int_0^t H^{1-\eta}(u)du  \times H^2(t) \leq \rho$ and thus
\[
H(t) \leq \frac{\rho^{\eta/2} H^{1-\eta}(t)}{ \left(\int_0^t H^{1-\eta}(u)du \right)^{\eta/2}}, \;\; \hbox{and by integration} \;\; \log(\alpha(t)) \leq \frac{\rho^{\eta/2}}{1-\eta/2} \left(\int_0^t H^{1-\eta}(u)du \right)^{1-\eta/2},
\]
or equivalently,
\begin{equation}\label{eqn.sympa}
 \log(\alpha(t))^{\frac{1}{1-\eta/2}}\leq \Gamma \times \int_0^t H^{1-\eta}(u)du, \;\; \hbox{where}\;\; \Gamma:=\left(\frac{\rho^{\eta/2}}{1-\eta/2} \right)^{\frac{1}{1-\eta/2}}.
\end{equation}
Note that  $y \mapsto y^{-2+\eta}$ is now integrable against $\nu$. Therefore, by Corollary 14 of \cite{anders} and with again the same notations as in the proof of Lemma \ref{lem.encadre}, we get that almost surely, when $s$ goes to infinity 
\[
\int_{0}^s \frac{du}{H^2(t_u)\dot{t}_u^2} \leq 2 \int_{0}^s \frac{du}{v_u^2} \leq 2 \int_{0}^s \frac{du}{y_u^2} = o\left( s^{\frac{1}{1-\eta/2}}\right).
\]
In particular, using the upper bound (\ref{eqn.sympa}), since $\log(\alpha(t_s)) = \frac{d-1}{2} \sigma^2 s +o(s)$ almost surely by Lemma \ref{lem.control.exp}, we get that almost surely for $s$ large enough
\[
\int_{0}^s \frac{du}{H^2(t_u)\dot{t}_u^2} \leq  \log(\alpha(t_s))^{\frac{1}{1-\eta/2}} \leq \Gamma \int_0^{t_s} H^{1-\eta}(u)du.
\]
Using this new upper bound in Equation (\ref{eqn.unsurxdeux}), we deduce that almost surely for $s$ large enough
\[
 D_s \leq  \Gamma \left[ \int_0^{t_s} H^{1-\eta}(u)du \times H^2(t_s) - 2 \int_0^s \left(\int_0^{t_s} H^{1-\eta}(u)du  \right)  H(t_u)H'(t_u) \dot{t}_u du\right]
\]
and a last integration by parts gives
\[
D_s \leq \Gamma \left[   \int_0^s  H^{1-\eta}(t_u)  H^2(t_u)) \dot{t}_u du\right] = \Gamma \left[ \int_{t_0}^{t_s}  H^{3-\eta}(u) du,\right]
\]
hence $D_s$ converges almost surely when $s$ goes to infinity and so does the clock $C_s$. 
\end{proof}

\begin{rem}
In the critical case i.e. if $d=3$ and $H^3 \in \mathbb L^1$ but $H^3 \notin \mathbb L^{1^-}$, the clock $C_s$ goes almost surely to infinity with $s$. Indeed, if $H^3 \in \mathbb L^1$, we know from Proposition 
\ref{pro.H3ssi} that $\dot{t}_s$ is almost surely transient. Then, with the same notations as in the proof of Lemma \ref{lem.encadre}, if $\varepsilon>0$, there exists a random proper time $s_0=s_0(\varepsilon,\omega)$ such that for all $s>s_0$ we have almost surely 
\[
 \frac{H^2(t_s)}{2}  \left( 1 +\frac{H'(t_s)}{H^2(t_s)}\right) \leq \frac{\varepsilon}{2} \quad \hbox{and} \quad  (d-1)\s^2\frac{  H^2(t_s) v_s }{v^2_s-H^2(t_s)} =(d-1)\s^2 \frac{  H(t_s) (\dot{t}_s +  \sqrt{\dot{t}_s^2-1})}{ (\dot{t}_s +  \sqrt{\dot{t}_s^2-1})^2_s-1} \leq  \frac{\varepsilon}{2}. 
\]
For $s>s_0$, we have thus, $v_s \leq v_s^{\varepsilon}$ almost surely, where $v^{\varepsilon}_{s_0}=v_{s_0}$ and $v_s^{\varepsilon}$ is solution of the following stochastic differential equation
\[
d v_s^{\varepsilon} =  - \frac{|v_s^{\varepsilon}|^2}{2}   ds + \frac{3 \sigma^2}{2}   v_s^{\varepsilon} ds  +\varepsilon ds + \sigma  v_s^{\varepsilon} d B_s. 
\]
The process $v_s^{\varepsilon}$ is positive and ergodic, its invariant probability measure is given by 
\[
\mu_{\varepsilon}(dx)= \frac{ \displaystyle{x \exp \left( -\frac{x}{\sigma^2} - \frac{2 \varepsilon}{\sigma^2 x}\right)}}{\displaystyle{ \int_0^{+\infty} x \exp \left(-\frac{x}{\sigma^2} - \frac{2 \varepsilon}{\sigma^2 x}\right) dx}}.
\]
For all $\varepsilon>0$, the function $x\mapsto 1/x^2$ is integrable against $\mu_{\varepsilon}$. Therefore, by the ergodic Theorem,  we have almost surely 
\[
 \lim_{s \to +\infty} \frac{1}{s}  \int_{s_0}^s \frac{du }{|v_u^{\varepsilon}|^2} = \frac{ \displaystyle{\int_{0}^{+\infty} x^{-1} e^{-\frac{x}{\sigma^2} - \frac{2 \varepsilon}{\sigma^2 x}}dx}}{\displaystyle{ \int_0^{+\infty} x e^{-\frac{x}{\sigma^2} - \frac{2 \varepsilon}{\sigma^2 x}} dx}}:=\gamma_{\varepsilon}.
\]
and thus, since $ H(t_s)\dot{t}_s \leq v_s \leq v_s^{\varepsilon}$ for $s$ large enough
\[
\liminf_{s \to +\infty} \frac{1}{s} \int_{s_0}^s \frac{du }{H^2(t_u)\dot{t}_u^2} \geq  \liminf_{s \to +\infty} \frac{1}{s} \int_{s_0}^s \frac{du }{v_u^2} \geq \liminf_{s \to +\infty} \frac{1}{s}  \int_{s_0}^s \frac{du }{|v_u^{\varepsilon}|^2} =\gamma_{\varepsilon}.
\]
Now, by the monotone convergence Theorem, $\gamma_{\varepsilon}$ goes to infinity when $\varepsilon$ goes to zero, so that
\[
  \log(\alpha(t_s)) =O(s) = o\left( \int_{s_0}^s \frac{du }{|v_u^{\varepsilon}|^2}  \right), \quad \hbox{and thus} \quad \log(\alpha(t_s)) = o \left( \int_0^s\frac{du}{H^2(t_u)\dot{t}_u^2}  \right).
\]
Injecting this estimate in Equation (\ref{eqn.unsurxdeux}) and integrating by parts again, we deduce that $D_s$ goes to infinity with $s$ almost surely, hence the result.
\end{rem}


\subsection{Study of the spatial components}\label{sec.spatial}
Having identified the asymptotic behavior of the temporal sub-diffusion $(t_s, \dot{t}_s)$, we can now give the proofs of Theorems \ref{theo.causal}, \ref{theo.causalderive} and \ref{theo.cercasymp} concerning the spatial components $(x_s, \dot{x}_s) \in T M$.

\subsubsection{Convergence to the causal boundary}
\smallskip
We first prove Theorem \ref{theo.causal}, i.e. the convergence of the projection $\xi_s \in \MM$ to the causal boundary $\pa \MM_c^+$.

\begin{proof}[Proof of Theorem \ref{theo.causal}]
Since Robertson-Walker space-times are globally hyperbolic (and a fortiori strongly causal), by definition of the causal boundary, we are left to show that the relativistic diffusion paths are inextendible. From Lemma \ref{LEM.UNIEXITEMP}, this is indeed the case, since the lifetime of the diffusion is precisely $\tau=\inf \{s >0, \; t_s = T\}$.
\end{proof}

We give now a concrete geometric description of this almost-sure convergence, showing in particular that $\xi_s \in \MM$ converges in fact to $\pa \MM_c^+ \backslash \{ i^+\}$ except in the case where $M=\mathbb S^d$ and $I_+(\a)=+\infty$, where $\xi_s \in \MM$ converges to $i^+$. Let us first consider the general case of space-times with finite horizon.

\begin{prop}\label{pro.horizon}
Let $\mathcal M:=(0, T) \times_{\alpha} M$ be a RW space-time satisfying the hypotheses of Sect. \ref{sec.RW} and such that $I_+(\a)<+\infty$. Let $(\xi_s, \dot{\xi}_s)=(t_s, x_s, \dot{t}_s, \dot{x}_s)$ be the relativistic diffusion starting from $(\xi_0, \dot{\xi}_0) \in T^1_+ \MM$. Then almost surely, when $s$ goes to $\ta=\inf \{ s>0, \, t_s= T\}$, we have naturally $t_s \to T$ and $x_s$ converges to a random point $x_{\ii} \in M$.
\end{prop}

\begin{proof} From Equation (\ref{eqn.pseudorw}), we have  
$|\dot{x}_s|^2=h(\dot{x}_s,\dot{x}_s)=a^2_s/\a^4(t_s)=(\dot{t}^2_s-1)/\a^2(t_s)\leq \dot{t}^2_s /\a^2(t_s)$. One thus deduce that for all $0\leq s < \ta$:
$$\textrm{dist}(x_s,x_0) \leq \int_0^s |\dot{x}_u |du  \leq \int_{t_0}^{t_s} \frac{du}{\a(u)}.$$
When $I_+(\alpha)<+\infty$ and $s$ goes to $\ta$, the last integral is almost surely convergent. The total variation of $x_s$ is thus almost surely convergent and it converges to a random variable $x_{\infty} \in M$. By definition of the explosion time $\tau$, $t_s$ goes to $T\leq +\infty$, therefore the projection $\xi_s=(t_s, x_s)$ converges almost surely to the random point $(T, x_{\infty}) \in \pa \MM_c^+$ according to the description of the causal boundary given in Example \ref{exa.cbound}.
\end{proof}

Let us now concentrate on the spatially flat case with infinite horizon i.e. $M=\mathbb R^d$ and $I_+(\a)=+\infty$.
\begin{prop}\label{pro.eudlicausal}
Let $\mathcal M:=(0, +\infty) \times_{\alpha} \mathbb R^d$ be a RW space-time satisfying the hypotheses of Sect. \ref{sec.RW} such that $I_+(\a)=+\infty$. Let $(\xi_s, \dot{\xi}_s)=(t_s, x_s, \dot{t}_s, \dot{x}_s)$ be the relativistic diffusion starting from $(\xi_0, \dot{\xi}_0) \in T^1_+ \MM$. Then almost surely, when $s$ goes to $\ta=\inf \{ s>0, \, t_s= T\}$, $(t_s, x_s)$ goes to infinity in a random preferred direction along a random hypersurface according to the second point of Example \ref{exa.cbound}.
\end{prop}

\begin{proof} First remark that if $I_+(\a)=+\infty$, the expansion is necesseraly at most polynomial (in fact sublinear). In the case $\mathcal M:=(0, T) \times_{\alpha} \mathbb R^d$, the system of stochastic differential equations satisfied by the global relativistic diffusion is the system (\ref{eqn.flj.eucli}). From It\^o's formula, one easily sees that the process $(t_s, \dot{t}_s, \dot{x}_s/|\dot{x}_s|)$ is itself a diffusion process, satisfying 
\begin{equation}\label{eds.eucli.cartesien}
\left \lbrace \begin{array}{l}
\ds{d t_s=\dot{t}_s ds,} \\
\\
\ds{d \dot{t}_s = - H(t_s) \left( \dot{t}_s^2-1\right) ds + \frac{d \s^2}{2} \dot{t}_s ds + d M^{\dot{t}}_s}, \\
\\
\ds{d \,  \frac{\dot{x}_s^i}{|\dot{x}_s|} = - \frac{d-1}{2} \frac{\s^2}{\dot{t}_s^2-1} \times  \frac{\dot{x}_s^i}{|\dot{x}_s|}\, ds + d M^{\dot{x}^i/|\dot{x}|}_s,}\end{array} \right. 
\end{equation}
\[ \hbox{with} \quad 
 \left \lbrace \begin{array}{l}
\ds{d \langle M^{\dot{t}}, \, M^{\dot{t}} \rangle_s = \s^2 \left( \dot{t}_s^2-1 \right) ds,} \\
\\
\ds{d \langle M^{\dot{t}}, \, M^{\dot{x}^i/|\dot{x}|} \rangle_s = 0,} \\
\\
\ds{d \langle M^{\dot{x}^i/|\dot{x}|},  M^{\dot{x}^j/|\dot{x}|} \rangle_s = \frac{\s^2}{\dot{t}_s^2-1}  \left(\delta_{ij} -  \frac{\dot{x}_s^i}{|\dot{x}_s|}\frac{\dot{x}_s^j}{|\dot{x}_s|}\right)ds}.
\end{array}\right. 
\]
Fix $s_0>0$ and consider the process $(\Theta_s)_{s \geq s_0}=(\Theta^1_s, \ldots, \Theta^d_s)_{s \geq s_0}$  defined as:
$$ \Theta^i_{C_s}:= \frac{\dot{x}_s^i}{|\dot{x}_s|}, \quad \textrm{where} \quad C_s:= \s^2 \int_{s_0}^s \frac{du}{\dot{t}_u^2-1} du =\s^2 \int_{s_0}^s \frac{\a^2(t_u)}{a^2_u} du .$$
Then $\Theta_s$ is solution of the stochastic differential equation:
$$ \ds{d \,\Theta^i_s   = -  \frac{d-1}{2} \Theta^i_s \, ds + d M^{\Theta^i}_s,} \quad \textrm{with} \quad \ds{d \langle M^{\Theta^i}, M^{\Theta^j}  \rangle_s=\left(\delta_{ij} - \Theta^i_s \Theta^j_s\right)ds}.$$
In other words, $\Theta_s$ is a standard spherical Brownian motion on $\mathbb S^{ d-1 }$, and $\dot{x}_s/|\dot{x}_s|$ is thus a time-changed spherical Brownian motion
where the clock $C_s$ is precisely the one introduced in Sect. \ref{sec.clock} above.
Since the expansion is at most polynomial, from Corollary \ref{cor.clock}, $C_s$ is almost surely convergent.
Hence, when $s$ goes to $\tau$ the normalized derivative $\dot{x}_s/|\dot{x}_s|$  converges almost surely to a random point $\Theta_{\infty} \in \mathbb S^2$.
Without loss of generality, one can suppose that $x_0=0$. For all $s<\ta$, we have thus
\[ 
\begin{array}{ll} x_s & = \displaystyle{\int_0^s \dot{x}_u du = \int_0^s \frac{\dot{x}_u}{|\dot{x}_u|} \times \frac{a_u}{\a^2(t_u)} du} \\
\\
&  = \displaystyle{\Theta_{\ii} \int_0^s \frac{a_u}{\a^2(t_u)} du + \int_0^s \left( \frac{\dot{x}_u}{|\dot{x}_u|}-\Theta_{\ii} \right) \times \frac{a_u}{\a^2(t_u)} du},
\end{array}
\]
taking the scalar product with $\Theta_{\ii}$, we get
\begin{equation}\label{decompx}\langle x_s, \Theta_{\ii} \rangle = \int_0^s \frac{a_u}{\a^2(t_u)} du +  \int_0^s \left \langle \left( \frac{\dot{x}_u}{|\dot{x}_u|}-\Theta_{\ii} \right), \; \Theta_{\ii} \right \rangle \times \frac{a_u}{\a^2(t_u)} du.\end{equation}
The first term of the right hand side can be written 
\[
\int_0^s \frac{a_u}{\a^2(t_u)} du= \int_{t_0}^{t_s} \frac{du}{\a(u)} - \int_0^s \frac{du}{a_u + \sqrt{a^2_u +\a^2(t_u)}}.
\]
From the study of the temporal sub-diffusion, the integral $\int_0^s du /(a_u + \sqrt{a^2_u +\a^2(t_u)})$ converges almost surely when $s$ goes to $\ta$.
Let us now show that the second term of the right hand side of (\ref{decompx}) converges almost surely when $s$ goes to $\ta$. From the beginning of the proof, we know that the process $\dot{x}_s/|\dot{x}_s|$ is a time-changed spherical Brownian motion. Namely, there exists a standard Brownian motion $W$ of dimension $d$  such that
$$ d  \frac{\dot{x}_s}{|\dot{x}_s|} = -  \frac{d-1}{2}  \s^2  \frac{\a^2(t_s)}{a^2_s} \frac{\dot{x}_s}{|\dot{x}_s|} ds +  d M^{\dot{x}/|\dot{x}|}_s, $$
with
$$ d M^{\dot{x}/|\dot{x}|}_s = \s \times \frac{\a(t_s)}{a_s} \times \left( d W_s -  \frac{\dot{x}_s}{|\dot{x}_s|} \times \left \langle \frac{\dot{x}_s}{|\dot{x}_s|}, d W_s  \right \rangle\right).$$
Integrating the last equation between $s$ and $\ta$, we get
\begin{eqnarray}\label{eqnxmoinstheta}\Theta_{\ii} - \frac{\dot{x}_s}{|\dot{x}_s|} & = & - \frac{d-1}{2}  \s^2  \int_s^{\ta} \frac{\a^2(t_u)}{a^2_u} \frac{\dot{x}_u}{|\dot{x}_u|} du \\ & & - \s \int_s^{\ta} \frac{\a(t_u)}{a_u} \times \left( d W_u -  \frac{\dot{x}_u}{|\dot{x}_u|} \times \left \langle \frac{\dot{x}_u}{|\dot{x}_u|}, d W_u  \right \rangle\right),\nonumber
 \end{eqnarray}
then taking the scalar product with $\Theta_{\ii}$: 
\begin{eqnarray}\label{eqnxmoinsthetascal}\left \langle \Theta_{\ii} - \frac{\dot{x}_s}{|\dot{x}_s|}, \; \Theta_{\ii}  \right\rangle & = & - \frac{d-1}{2}  \s^2  \int_s^{\ta} \frac{\a^2(t_u)}{a^2_u} \left \langle \frac{\dot{x}_u}{|\dot{x}_u|}, \Theta_{\ii} \right \rangle du \\ 
& & -\s \int_s^{\ta} \frac{\a(t_u)}{a_u}\left \langle \Theta_{\ii} - \frac{\dot{x}_s}{|\dot{x}_s|}, \; \Theta_{\ii}  \right\rangle \left \langle \Theta_{\ii}, \; d W_u \right\rangle \nonumber \\
& & - \s \int_s^{\ta} \frac{\a(t_u)}{a_u}\left \langle \Theta_{\ii}, \; \frac{\dot{x}_s}{|\dot{x}_s|} \right\rangle \left \langle \frac{\dot{x}_s}{|\dot{x}_s|} -\Theta_{\ii}, \; d W_u \right\rangle. \nonumber
 \end{eqnarray}

\noindent
From the law of the iterated logarithm, for all $\e>0$, we have almost surely when $s$ goes to $\ta$:
$$\left|\int_s^{\ta} \frac{\a(t_u)}{a_u} \times \left( d W_u -  \frac{\dot{x}_u}{|\dot{x}_u|} \times \left \langle \frac{\dot{x}_u}{|\dot{x}_u|}, d W_u  \right \rangle\right) \right| = o \left( \left[ \int_s^{\ta} \frac{\a^2(t_u)}{a^2_u}  du \right]^{1/2-\e}  \right).$$
From Equation (\ref{eqnxmoinstheta}), one deduce that almost surely when $s$ goes to $\ta$: 
$$\left |\Theta_{\ii} - \frac{\dot{x}_s}{|\dot{x}_s|} \right|  =o \left( \left[ \int_s^{\ta} \frac{\a^2(t_u)}{a^2_u}  du \right]^{1/2-\e}  \right).$$
Injecting this estimate in (\ref{eqnxmoinsthetascal}) and applying the law of the iterated logarithm again, we obtain that for all $\e>0$, when $s$ goes to $\ta$:
$$\left|\left \langle \Theta_{\ii} - \frac{\dot{x}_s}{|\dot{x}_s|}, \; \Theta_{\ii}  \right\rangle \right|  =  o \left( \left[ \int_s^{\ta} \frac{\a^2(t_u)}{a^2_u }du\left( \int_u^{\ta} \frac{\a^2(t_v)}{a^2_v}  dv\right)^{1-2 \e} \right]^{1/2-\e}  \right). $$
From the asymptotic estimates obtained in Proposition \ref{PRO.ASYMPC}, we conclude that when $s$ goes to $\ta$:   
$$\int_0^{s} \left \langle \Theta_{\ii} - \frac{\dot{x}_u}{|\dot{x}_u|}, \; \Theta_{\ii}  \right\rangle \frac{a_u}{\a^2(t_u)} du < +\ii.$$
We have thus shown that, almost surely when $s$ goes to $\tau$, the process 
$ \delta_s:=\int_{t_0}^{t_s} \frac{du}{\a(u)} - \langle x_s, \Theta_{\ii} \rangle$
converges to a limit $\delta_{\infty} \in \mathbb R^+$, in particular since $I_+(\a)=+\infty$, $|x_s|$ goes to infinity with $s$. Therefore, with the same notations as in Example \ref{exa.cbound}, the projection $x_s \in \mathbb R^d$ goes to infinity in the random direction $\Theta_{\infty}$ and $(t_s, x_s)$ goes to infinity in the same direction along the hypersurface $\Sigma(\delta_{\infty}, \Theta_{\infty})$.
\end{proof}

\begin{rem}
Note that in the flat case $\alpha(t) \equiv \alpha$, we recover the long-time asymptotic behavior derived in \cite{ismael}. Moreover, let us emphasize that if the expansion is ``really'' polynomial in the sense that $H(t) \times t $ goes to $c \in (0,1]$ when $t$ goes to infinity or equivalently if 
\[
 \lim_{t \to +\infty} \frac{\log \left( \alpha(t)\right)}{\log \left(\int^t \alpha(u) du\right)}=\frac{c}{1+c} <1,
\]
the path $(t_s, x_s)$ is not only asymptotic to the random hypersurface $\Sigma(\delta_{\infty}, \Theta_{\infty})$, but it is asymptotic to a random curve in the sense that
\[
x_s - \Theta_{\infty} \int_{t_0}^{t_s} \frac{du}{\a(u)}.
\]
converges almost surely. 
Indeed, from the proof above, we have 
\[ 
x_s  = \displaystyle{ \Theta_{\ii} \int_{t_0}^{t_s} \frac{du}{\a(u)}  + \int_0^s \left( \frac{\dot{x}_u}{|\dot{x}_u|}-\Theta_{\ii} \right) \times \frac{a_u}{\a^2(t_u)} du } + \displaystyle{\Theta_{\ii} \int_0^s \frac{du}{a_u + \sqrt{a^2_u +\alpha^2(t_u)}}},
\]
where the last term is almost surely convergent. Otherwise, we have seen that 
\[
\left |\Theta_{\ii} - \frac{\dot{x}_s}{|\dot{x}_s|} \right|  =o \left( \left[ \int_s^{\ta} \frac{\a^2(t_u)}{a^2_u}  du \right]^{1/2-\e}  \right).
\]
If $c \in (0,1)$, combining this estimate with the ones of Proposition \ref{PRO.ASYMPC}, we get that 
\[ \int_0^s \left( \frac{\dot{x}_u}{|\dot{x}_u|}-\Theta_{\ii} \right) \times \frac{a_u}{\a^2(t_u)} du  \]
is almost surely convergent, hence the result.
\end{rem}

We consider now the case of a negatively curved fibre with infinite horizon i.e. $M=\mathbb H^d$ and $I_+(\a)=+\infty$.
\begin{prop}\label{pro.hypercausal}
Let $\mathcal M:=(0, +\infty) \times_{\alpha} \mathbb H^{d}$ be a RW space-time satisfying the hypotheses of Sect. \ref{sec.RW} such that $I_+(\a)=+\infty$. Let $(\xi_s, \dot{\xi}_s)=(t_s, x_s, \dot{t}_s, \dot{x}_s)$ be the relativistic diffusion starting from $(\xi_0, \dot{\xi}_0) \in T^1_+ \MM$. Let us write $x_s$ in polar coordinates, namely $x_s=(\sqrt{1+r_s^2}, r_s \theta_s)$ with $r_s >0$ and $\theta_s \in \mathbb S^{d-1}$. Then almost surely, when $s$ goes to $\ta=\inf \{ s>0, \, t_s= T\}$, we have 
\begin{enumerate}
\item The angle $\theta_s$ converges to a random point $\theta_{\infty} \in \mathbb S^{d-1}$;
\item The projection $x_s$ converges to the random hyperplane 
$$\Pi(\theta_{\infty}):=\{ x \in \mathbb H^{d} \subset \mathbb R^{1,d},\; q(x, (1,\theta_{\infty}))=0\},$$
where $q$ is the usual Minkowskian scalar product;
\item The radial process is transient and there exists a random positive real number $\delta_{\infty}$ such that
$$\delta_s:=\int_{t_0}^{t_s} \frac{du}{\alpha(u)} - \ash(r_s) \to \delta_{\infty}.$$ 
\end{enumerate}
\end{prop}

\begin{proof}
In polar coordinates i.e. if $x=(\sqrt{1+r^2}, r \theta)$ with $r>0$ and $\theta \in \mathbb S^{d-1}$, the normalized spatial derivative $\dot{x}_s/|\dot{x}_s|$ reads 
$$ \ds{\frac{\dot{x}_s}{|\dot{x}_s|}= \left(\frac{c_s}{a_s} \times r_s, \, \frac{c_s}{a_s} \times \sqrt{1+r^2_s} \times \t_s + \frac{\r_s}{r_s} \times \frac{\dot{\t}_s}{|\dot{\t}_s|}\right)},$$
where 
$$\begin{array}{l}\begin{array}{llll}
\ds{a_s= \a(t_s) \sqrt{\dot{t}_s^2-1}}, & \ds{c_s:=\frac{\a^2(t_s) \dot{r}_s}{\sqrt{1+ r^2_s}}} &
  \ds{ b_s:= \a^2(t_s) r^2_s |\dot{\t}_s|}, & \ds{\r_s:=b_s/a_s}.\end{array} \end{array}
$$
From the pseudo-norm relation (\ref{eqn.pseudorw}), we have moreover
\begin{equation} \label{eqn.pseudo2}\ds{\;\; a^2_s=\frac{b^2_s}{r^2_s} + c^2_s}, \quad i.e. \quad \ds{1 =\frac{\r^2_s}{r^2_s}+\frac{c^2_s}{a^2_s}}.\end{equation}
Starting from Equation (\ref{eqn.flj}) in polar coordinates, a straightforward calculation using It\^o's formula shows that the process $c_s/a_s$ satisfies, for $s \geq s_0$ 
$$\frac{c_s}{a_s} - \frac{c_{s_0}}{a_{s_0}} =  I_s - J_s  + M^{c/a}_s, $$
with
$$I_s:=  \int_{s_0}^s \frac{\r^2_u }{r^2_u} \times \sqrt{\frac{1}{r^2_u}+1} \times \frac{a_u}{\a^2(t_u)}du, \qquad J_s:=\frac{d-1}{2} \s^2 \int_{s_0}^s \frac{\a^2(t_u)}{a^2_u} \frac{c_u}{a_u} du,$$
and $M^{c/a}_s$ is a local martingale whose braket is given by 
$$ \langle M^{c/a}\rangle_s:= \s^2 \int_{s_0}^s  \frac{\a^2(t_u)}{a_u^2} \frac{\r_u^2}{r_u^2 }du.$$
As above, since $I_+(\a)=+\infty$ the expansion is necessarily at most polynomial and Corollary \ref{cor.clock} ensures that the clock 
$$C_s=\sigma^2 \int_{s_0}^s \frac{\a^2(t_u)}{a^2_u} du = \sigma^2\int_{s_0}^s \frac{du}{\dot{t}^2_u-1} du $$ 
converges almost surely when $s$ goes to $\tau$. Since $|c_s/a_s|$ and $\rho_s/r_s$ are bounded by one, both processes $J_s$ and $M^{c/a}_s$ also converge almost surely when $s$ goes to $\tau$. Again, $|c_s/a_s|$ being bounded by one, the non-decreasing integral $I_s$ is also convergent and so does the process $c_s/a_s$. We claim that necesseraly $\lim_{s \to \tau}c_s/a_s =1$ almost surely. Indeed, since  $I_+(\a)=+\infty$, from the study of the temporal diffusion $(t_s, \dot{t}_s)$ we know that 
$$ \int_{s_0}^s \frac{a_u}{\a^2(t_u)}du= \int_{s_0}^s \frac{\sqrt{\dot{t}_u^2-1}}{\a(t_u)}du= \int_{\ds{t_{s_0}}}^{t_s} \frac{du}{\a(u)}+o\left(   \int_{\ds{t_{s_0}}}^{t_s} \frac{du}{\a(u)}   \right) \longrightarrow +\infty.$$
Let us define $A:=\{ \omega, \, \lim_{s \to \tau}c_s^2/a_s^2(\omega)<1\}$. From the pseudo-norm relation (\ref{eqn.pseudo2}), on the set $A$, we have $\lim_{s \to \tau}\rho^2_s/r^2_s(\omega)>0$. Therefore, on the set $A$, the non-decreasing integral $I_s$ is not convergent, and from above, we conclude that $A$ is a negligeable set. We have thus shown that $\lim_{s \to \tau}c_s^2/a_s^2 =1$ almost surely. To conclude, we remark that the limit $\lim_{s \to \tau}c_s/a_s$ is necesseraly positive, otherwise $r_s$ would tend to $-\infty$. As $s$ goes to $\tau$, we thus have almost surely
$$  \frac{\dot{r}_s}{\sqrt{1+ r^2_s}}\simeq \frac{a_s}{\a^2(t_s)}, \;\; \hbox{and by integration } \;\; \ash(r_s) \simeq \int_.^s \frac{a_u}{\a^2(t_u)}du .$$
Since $\lim_{s \to \tau}c_s/a_s=1$ almost surely, from the relation (\ref{eqn.pseudo2}), we have $ \lim_{s \to \tau}\rho_s/r_s=0$ almost surely, that is 
$ \lim_{s \to \tau}b_s/ a_s r_s=0$ or equivalently
$$ |\dot{\theta}_s | = o \left( \frac{a_s}{\alpha^2(t_s) r_s} \right) = o \left( \frac{a_s}{\alpha^2(t_s) } \exp \left( -\int_.^s \frac{a_u}{\a^2(t_u)}du\right)\right).$$
Therefore, when $s$ goes to $\tau$, the angle $\theta_s$ converges almost surely to a random point $\theta_{\infty} \in \mathbb S^{d-1}$ and integrating the last estimate, we have
\[
| \theta_{\infty}-\theta_s | = o\left( \exp \left( -\int_.^s \frac{a_u}{\a^2(t_u)}du\right)\right). 
\]
 The Minkowskian scalar produit between $x_s$ and $(1,\theta_{\infty})$ can be written as
\[
q\left(x_s, (1,\theta_{\infty})\right) = \left(  \sqrt{1+r^2_s}-r_s \right)+ r_s \langle  \theta_{\infty}-\theta_s, \theta_s \rangle.
\]
Since $r_s$ goes almost surely to infinity with $s$, the first term of the right hand side vanishes at infinity, and so does the second term from the above estimates on $r_s$ and $| \theta_{\infty}-\theta_s |$.
We are left to show that $\delta_s$ converges almost surely. To see this, we write
\[
 \frac{\dot{r}_s}{\sqrt{1+ r^2_s}}= \frac{a_s}{\a^2(t_s)} - \left( 1- \frac{c_s}{a_s}\right) \frac{a_s}{\a^2(t_s)} =  \frac{\dot{t}_s}{\a(t_s)} - \frac{1}{a_s+\sqrt{a^2_s +\alpha^2(t_s)}}- \left( \frac{1}{1+ \frac{c_s}{a_s}} \frac{\r^2_s }{r^2_s} \frac{a_s}{\a^2(t_s)}\right) .
\]
From the study of the temporal diffusion, we know that almost surely
\[ 
\int_0^{+\infty} \frac{du}{a_u+\sqrt{a^2_u +\alpha^2(t_u)}}<+\infty.
\]
Moreover, since $c_s/a_s\to 1$, we have for $s$ large enough
\[
\int_0^s \frac{1}{1+ \frac{c_u}{a_u}} \frac{\r^2_u }{r^2_u} \frac{a_u}{\a^2(t_u)}du \leq \int_0^s \frac{\r^2_u }{r^2_u} \frac{a_u}{\a^2(t_u)}du,
\]
which is almost surely convergent when $s$ goes to infinity since the integral $I_s$ is. Thus, we can conclude that $\delta_s$ converges almost surely to
\[
\delta_{\infty} =\delta_0 + \int_0^{+\infty} \frac{du}{a_u+\sqrt{a^2_u +\alpha^2(t_u)}}+ \int_0^{\infty} \frac{1}{1+ \frac{c_u}{a_u}} \frac{\r^2_u }{r^2_u} \frac{a_u}{\a^2(t_u)}du.
\]
\end{proof}

\begin{rem}
In Propositions \ref{pro.horizon}, \ref{pro.eudlicausal} and \ref{pro.hypercausal}, we gave the geometric description of the convergence to the causal boundary in the case where $I_+(\a)<+\infty$ and $I_+(\a)=+\infty$ and $M= \mathbb R^d$ and $M=\mathbb H^d$. Thus, the only remaining case is the case $I_+(\a)=+\infty$ and $M=\mathbb S^d$ whose description is implicit at the end of the proof of Theorem \ref{theo.cercasymp} below.
\end{rem}

\subsubsection{Proof of Theorem \ref{theo.causalderive}} 
\label{sec.tangentunitaire}
\smallskip
In this section, we give the proof of Theorem \ref{theo.causalderive} concerning the asymptotic behavior of the normalized spatial derivative $\dot{x}_s/|\dot{x}_s| \in T^1_{x_s} M$ when $I_+(\a)<+\infty$.

\begin{proof}
In the case of interest, namely when $I_+(\a)<+\infty$, the convergence of the spatial projection $x_s$ was already obtained in Proposition \ref{pro.horizon}. 
Let us first prove the point $i)$ and $ii)$ of  Theorem \ref{theo.causalderive}, i.e. the convergence of $\dot{x}_s/|\dot{x}_s|$ if $T<+\infty$ or if $T=+\infty$ and the expansion is at most polynomial or subexponential with $H^3 \in \mathbb L^1$.
For this, starting from Equation (\ref{eqn.flj}), we explicit the stochastic differential equations system satisfied by $(t_s, \dot{t}_s, x_s, \dot{x}_s/|\dot{x}_s|)$ in Cartesian coordinates. 
In the Euclidean case $M=\mathbb R^d$, this system is nothing but the system (\ref{eds.eucli.cartesien}) obtained in the proof of Proposition \ref{pro.eudlicausal}. 
In a synthetic way, in the hyperbolic case $M=\mathbb H^d$ viewed as the half-sphere of the Minkowski space $\mathbb R^{1,d}$ with Cartesian coordinates $(x^0, x^1, \ldots, x^d)$, or in the spherical case $M=\mathbb S^d$ viewed as the sphere of the Euclidean space $\mathbb R^{d+1}$ with Cartesian coordinates $(x^0, x^1, \ldots, x^d)$, the system can be written
\begin{equation}\label{eds.hypersphere.cartesien}
\begin{array}{ll}
& \left \lbrace \begin{array}{l}
\ds{d t_s=\dot{t}_s ds,} \\
\\
\ds{d \dot{t}_s = - H(t_s) \left( \dot{t}_s^2-1\right) ds + \frac{d \s^2}{2} \dot{t}_s ds + d M^{\dot{t}}_s}, \\
\\
\ds{d \,  \frac{\dot{x}_s^{\mu}}{|\dot{x}_s|} = - \kappa \, x^{\mu}_s \times \frac{\sqrt{\dot{t}_s^2-1}}{\a(t_s)} ds - \frac{d-1}{2} \frac{\s^2}{\dot{t}_s^2-1} \times  \frac{\dot{x}_s^{\mu}}{|\dot{x}_s|}\, ds + d M^{\dot{x}^{\mu}/|\dot{x}|}_s,}\end{array} \right. \\
\\
\textrm{with} 
& \left \lbrace \begin{array}{l}
\ds{d \langle M^{\dot{t}}, \, M^{\dot{t}} \rangle_s = \s^2 \left( \dot{t}_s^2-1 \right) ds,} \\
\\
\ds{d \langle M^{\dot{t}}, \, M^{\dot{x}^{\mu}/|\dot{x}|} \rangle_s = 0,} \\
\\
\ds{d \langle M^{\dot{x}^{\mu}/|\dot{x}|},  M^{\dot{x}^{\nu}/|\dot{x}|} \rangle_s = \frac{\s^2}{\dot{t}_s^2-1}  \left(\delta_{\mu \nu} +(\kappa -1)\delta_{\mu 0} \delta_{\nu 0}- \kappa \, x^{\mu}_sx^{\nu}_s-  \frac{\dot{x}_s^{\mu}}{|\dot{x}_s|}\frac{\dot{x}_s^{\nu}}{|\dot{x}_s|}\right)ds},
\end{array}\right. \end{array}\end{equation}
where $\kappa$ denotes the curvature of the space, namely $\kappa=-1$ if $M=\mathbb H^3$ and $\kappa=1$ when $M=\mathbb S^3$.
From both Equations (\ref{eds.eucli.cartesien}) and (\ref{eds.hypersphere.cartesien}), since $x_s$ is convergent and $I_+(\a)<+\infty$, it is clear that $\dot{x}_s/|\dot{x}_s|$ is almost surely convergent if and only if  the inverse of $\dot{t}_s^2-1$ is integrable in the neighborhood of $\tau$, i.e. if and only if the clock 
$C_s$  introduced in Sect. \ref{sec.clock} is almost surely convergent when $s$ goes to $\tau$. From Corollary \ref{cor.clock}, it is the case if $T<+\infty$ or if $T=+\infty$ and the expansion is at most polynomial or subexponential with $H^3 \in \mathbb L^1$ (or $H^3 \in \mathbb L^{1^{-}}$ if $d=3$), hence the result. \par

We now give the proof of point $ii)$ concerning the asymptotic behavior of $\dot{x}_s/|\dot{x}_s|$ when $T=+\infty$ and the expansion is exponential or subexponential with $H^3 \notin \mathbb L^1$. In the Euclidean case $M=\mathbb R^d$, from the proof of Proposition \ref{pro.eudlicausal}, $\dot{x}_s/|\dot{x}_s|$ is a time-changed spherical brownian motion parametrized by the clock $C_s$. From Corollary \ref{cor.clock}, we know that $C_s$ goes to infinity with $s$ almost surely, hence the result. For the two remaining cases 
$M=\mathbb H^d$ or $M=\mathbb S^d$, the proofs are very similar, so we will restrict ourself to the spherical case. Moreover, to simplify the expressions, we will suppose that $d=3$ but the proof applies verbatim for $d \geq 4$.
In the sequel, $\mathbb S^3$ is viewed as the sphere of the Euclidean space $\R^4$, therefore elements of $\mathbb S^3$ or $T_{.} \mathbb S^3$ can be seen as elements of $\mathbb R^4$.
Fix an orthonormal frame $e_0=(e_0^1, e_0^2, e_0^3)$ in the unitary tangent space $T_{x_0}^1 \S^3$, and denote by $e_s=(e_s^1, e_s^2, e_s^3)$ the frame of $T_{x_s}^1 \S^3 \subset \R^4$ obtained by deterministic parallel transport along the great circle joining $x_0$ to $x_s$. When $s$ goes to infinity, $e_s$ converges almost surely to a frame $e_{\infty}$ in $T_{x_{\ii}}^1 \S^3$. Since the path $x_s$ is $C^1$, of finite total variation, $e_s$ is also $C^1$ and if 
$\dot{e}^i_s:= d e_s^i /ds$, we have almost surely 
$$ \int_0^{+\ii} (||\dot{e}_s^1||+  ||\dot{e}_s^2||+ ||\dot{e}_s^3||) ds < +\ii, \; \hbox{where} \; ||\cdot|| \; \hbox{is the Euclidean norm in}\; \mathbb R^4. $$
Let us denote by $u^i_s$ the coordinates of $\dot{x}_s/|\dot{x}_s|$ in the frame $e_s$, that is $u^i_s:=\langle \dot{x}_s/|\dot{x}_s|, \, e_s^i\rangle,$ for $i=1,2,3.$
From Equation (\ref{eds.hypersphere.cartesien}) with $\kappa = 1$, the process $u_s=(u^1_s,u^2_s, u_s^3)$ satisfies 
\begin{equation} d u_s^i =  - \s^2 \times u_s^i \times \frac{\a^2(t_s)}{a^2_s} ds + \left \langle \frac{\dot{x}_s}{|\dot{x}_s|}, \, \dot{e}_s^i \right\rangle ds + d M^{u^i}_s ,\end{equation}
$$\hbox{where} \qquad  d \langle M^{u^i}, M^{u^j}\rangle_s = \s^2
\left( \d_{i j} - u^i_s u^j_s\right)\frac{\a^2(t_s)}{a^2_s} ds.$$
The martingales $M^{u^i}$ can be represented by 3-dimensional standard Brownian motion $W=(W^1,W^2, W^3)$ in the following way:
$$\begin{array}{l}\ds{ d M^{u^1 }_s = \s \times\frac{\a(t_s)}{a_s}  \times \left( u^3_s d W^2_s + u^2_s d W^3_s\right),} \\ \ds{d M^{u^2 }_s = \s \times\frac{\a(t_s)}{a_s}  \times \left( u^3_s d W^1_s - u^1_s d W^3_s\right),} \\
\ds{d M^{u^3 }_s = \s \times\frac{\a(t_s)}{a_s}  \times \left( -u^2_s d W^1_s - u^1_s d W^2_s\right).} \\
\end{array}
$$
Fix $\e>0$  and consider a large enough (random) proper time $s_{\e}$ so that, for all $s \geq s_{\e}$: 
\begin{equation} \int_{s}^{+\ii} (||\dot{e}_u^1||+ ||\dot{e}_u^2||+||\dot{e}_u^3||)du \leq \e^2 /(4 \times 36) \;\; \hbox{and} \;\; \sup_{s \geq s_{\e}} \sum_{i=1}^3||e^i_s-e^i_{\ii} || \leq \e/2.\label{eqn.borneetahyper} \end{equation}
Consider the process $v_s=(v^1_s, v^2_s,v^3_s)$ starting from $u_{s_{\e}}=(u^1_{s_{\e}} u^2_{s_{\e}}, u^3_{s_{\e}})$ and solution of the following equation, for $s \geq s_{\e}$:
$$d v_s^i =  - \s^2 \times v_s^i \times \frac{\a^2(t_s)}{a^2_s} ds + d M^{v^i}_s$$
where
$$\begin{array}{l}\ds{ d M^{v^1 }_s = \s \times\frac{\a(t_s)}{a_s}  \times \left( v^3_s d W^2_s + v^2_s d W^3_s\right),} \\ \ds{d M^{v^2 }_s = \s \times\frac{\a(t_s)}{a_s}  \times \left( v^3_s d W^1_s - v^1_s d W^3_s\right),} \\
\ds{d M^{v^3 }_s = \s \times\frac{\a(t_s)}{a_s}  \times \left( -v^2_s d W^1_s - v^1_s d W^2_s\right).} \\
\end{array}
$$
The processes $v_s^i$, $i=1,2,3$, are the coordinates of time-changed spherical Brownian motion in $\S^2$, the clock being given by  $s':=\s^2 \int^s (\a^2(t_u)/a^2_u) du$ which goes to infinity with $s$ almost surely from Corollary \ref{cor.clock}. A straightforward computation shows that $R^2_s:= |u^1_s - v^1_s|^2 + |u^2_s - v^2_s|^2+ |u^3_s - v^3_s|^2$ satisfies the following equation  for $s \geq s_{\e}$: 
\begin{equation}d R^2_s = 2 \sum_{i=1}^3  (u^i_s - v^i_s) \left \langle \frac{\dot{x}_s}{|\dot{x}_s|}, \, \dot{e}_s^i \right\rangle ds \label{eqn.umoinsvhyper}.\end{equation}
From (\ref{eqn.borneetahyper}), for $s \geq s_{\e}$, we thus have almost surely 
$$ R^2_s \leq  \e^2/36, \;\; \hbox{in particular} \;\; |u_s^i-v_s^i| \leq \e/6 \;\; \hbox{for} \;\; i=1,2,3.$$
Let us introduce the process $\Theta_s^{\e}$ defined for  $s \geq s_{\e}$ by $\Theta_s^{\e}:= \sum_{i=1}^3 v^i_s e^i_{\ii}$. It is a time-changed spherical Brownian motion in the unitary tangent space $T_{x_{\ii}}^1 \S^3 \approx \S^2$, parametrized by the clock
$\s^2 \int_{s_{\e}}^s \frac{\a^2(t_u)}{a^2_u} du,$ 
which goes to infinity with $s$. From the above estimates, we have almost surely for all $s \geq s_{\e}$, 
$$ \begin{array}{ll}\ds{\left|\left|\frac{\dot{x}_{s }}{|\dot{x}_{s}|}- \Theta_s^{\e}\right|\right|} & \ds{= \left|\left| \sum_{i=1}^3 u^i_s e^i_s - \sum_{i=1}^3 v^i_s e^i_{\ii} \right|\right|} \\
& \ds{=\left|\left| \sum_{i=1}^3(u^i_s-v^i_s)  e^i_{\ii} + \sum_{i=1}^3 u^i_s (e^i_s-e^i_{\ii}) \right|\right|}\\
& \ds{\leq \sum_{i=1}^3\left| u^i_s-v^i_s \right|  +  \sum_{i=1}^3\left|\left|e^i_s-e^i_{\ii} \right| \right|}\\
&  \ds{\leq \e/2 +\e/2 \leq \e,}
 \end{array}$$
hence the result.
\end{proof}

\subsubsection{Proof of Theorem \ref{theo.cercasymp}.}\label{sec.cercasymp}
\smallskip
Finally, we give the proof of Theorem \ref{theo.cercasymp} concerning the asymptotic behavior of the normalized spatial derivative $\dot{x}_s/|\dot{x}_s| \in T^1_{x_s} M$ in the case $I_+(\a)=+\infty$.

\begin{proof}
The Euclidean case $M=\mathbb R^d$ is again the easiest case:  we already know that $\dot{x}_s/|\dot{x}_s| $ is a time-changed spherical Brownian motion on $\mathbb S^{d-1}$. Since $I_+(\a)=+\infty$, the expansion is necesseraly at most polynomial and from Corollary \ref{cor.clock}, the clock $C_s$ converges almost surely when $s$ goes to $\tau$, hence the result. 

Let us now treat the hyperbolic case $M=\mathbb H^d \subset \mathbb R^{1,d}$. With the same notations as in the proof of Proposition \ref{pro.hypercausal}, we have 
\[
\ds{\frac{1}{|x^0_s|} \frac{\dot{x}_s}{|\dot{x}_s|}}  =\ds{  \frac{1}{\sqrt{1+r^2_s}}\frac{\dot{x}_s}{|\dot{x}_s|}}
= \ds{\left(\frac{c_s}{a_s} \times \frac{r_s}{\sqrt{1+r^2_s}}, \, \frac{c_s}{a_s}  \times \t_s + \frac{\r_s}{r_s \sqrt{1+r^2_s}} \times \frac{\dot{\t}_s}{|\dot{\t}_s|}\right)} .
\]
By Proposition \ref{pro.hypercausal}, we know that $r_s$ is almost surely transient and $\theta_s$ converges to $\theta_{\infty}$ so that 
\[\ds{\lim_{s \to \tau} \frac{1}{|x^0_s|} \frac{\dot{x}_s}{|\dot{x}_s|}}  =\ds{ \lim_{s \to \tau} \frac{1}{\sqrt{1+r^2_s}}\frac{\dot{x}_s}{|\dot{x}_s|}}
= \ds{\lim_{s \to \tau}\left(\frac{c_s}{a_s} \times \frac{r_s}{\sqrt{1+r^2_s}}, \, \frac{c_s}{a_s}  \times \t_s + \frac{\r_s}{r_s \sqrt{1+r^2_s}} \times \frac{\dot{\t}_s}{|\dot{\t}_s|}\right)} =\ds{(1, \theta_{\infty})}.
\]
Finally, we discuss the last and more surprising case where $M=\mathbb S^d$. We show that both $x_s$ and its normalized derivative $\dot{x}_s/|\dot{x}_s|$ asymptotically describe a random great circle in $\mathbb S^d$. Recall that if $\mathbb S^d$ is endowed with the global Cartesian coordinates of $\mathbb R^{d+1}$, the 
process $(x_s, \dot{x}_s/|\dot{x}_s|)$ satisfies the stochastic differential equations system:
\begin{equation}\label{eds.sphere.cartesien} \begin{array}{l}
\ds{d x_s = \frac{\dot{x}_s}{|\dot{x}_s|} \times \frac{a_s}{\a^2(t_s)} ds, \qquad d \frac{\dot{x}_s}{|\dot{x}_s|} = - x_s \times \frac{a_s}{\a^2(t_s)} ds + d \eta_s,}
\end{array}\end{equation}
with
$$ \ds{d \eta_s:= -\frac{d-1}{2} \s^2 \frac{\a^2(t_s)}{a^2_s} \times   \frac{\dot{x}_s}{|\dot{x}_s|} ds + d M^{\dot{x}/|\dot{x}|}_s,}$$
and where for $\m, \n$ in $\{0, 1, \ldots, d\}$:
$$
\ds{d \langle M^{\dot{x}^{\m}/|\dot{x}|},  M^{\dot{x}^{\n}/|\dot{x}|} \rangle_s =  \s^2 \frac{\a^2(t_s)}{a^2_s}\left(\delta_{\m \n} -  x^{\m}_s x^{\n}_s -  \frac{\dot{x}_s^{\mu}}{|\dot{x}_s|}\frac{\dot{x}_s^{\n}}{|\dot{x}_s|}\right)ds.}$$
Under this form, the system (\ref{eds.sphere.cartesien}) can be seen as an harmonic oscillator, perturbed by $d \eta_s$, and time-changed by the clock $ds' =  a_s ds /a^2(t_s)$. To simplify the expressions, let us introduce the notations $y_s:= \dot{x}_s/|\dot{x}_s|$,  $A_s:= \int_0^s \frac{a_u}{\a^2(t_u)} \, du$. Then, the complex process $z_s:= x_s +i y_s$ is solution to
$$d z_s = - i \, z_s d A_s + i d \eta_s.$$
Thus, for $0 \leq s < \ta$, we have: 
\begin{equation}z_s = z_0 e^{ -i A_s} + i e^{-i A_s} I_s, \quad \hbox{where} \quad I_s:=\int_0^s e^{i A_u} d \eta_u.\label{eqn.zsphere} \end{equation}
The integral $I_s$ decomposes into a sum $I_s=J_s+K_s$ where
$$J_s:= -\frac{d-1}{2} \s^2 \int_0^s e^{i A_u} \frac{\a^2(t_u)}{a^2_u} \times   y_u du, \quad K_s:= \int_0^s e^{i A_u} d M^{\dot{x}/|\dot{x}|}_u.$$
Under our hypotheses, the clock  $C_s= \int_0^s \frac{\a^2(t_u)}{a^2_u} du$ converges almost surely when $s$ goes to $\ta$. Thus, the total variation of $J_s$ and the quadratic variation of $K_s$ converge almost surely when $s$ goes to $\ta$. Consequently, $J_s, \,K_s$ and $I_s$ are almost surely convergent. Let us denote by $I_{\ii}$ the limit of the integral $I_s$. From Equation (\ref{eqn.zsphere}), when $s$ goes to $\ta$, we have:
$$z_s = \left( z_0  + i I_{\ii}\right) e^{-i A_s} - i e^{-i A_s} \left(I_{\ii} - I_s\right) =\left( z_0  + i I_{\ii}\right) e^{-i A_s}+o(1).$$ 
In other words, defining $U_{\ii}:= x_0- \Im(I_{\ii})$ and $V_{\ii}:= y_0+ \Re(I_{\ii})$, i.e.
$$\begin{array}{l}
\ds{U_{\ii} = x_0 - \int_0^{+\ii} \sin \left( \int_0^s \frac{a_u}{\a^2(t_u)} du\right) d \eta_s }, \\
\ds{V_{\ii} = y_0 + \int_0^{+\ii} \cos \left( \int_0^s \frac{a_u}{\a^2(t_u)} du\right) d \eta_s }, 
\end{array}$$
we obtain that almost surely, when $s$ goes to $\ta$: 
$$\begin{array}{l}
\ds{x_s = \cos \left( \int_0^s \frac{a_u}{\a^2(t_u)} du \right) U_{\ii}  + \sin \left( \int_0^s \frac{a_u}{\a^2(t_u)} du \right) V_{\ii} + o(1),} \\
\\
\ds{\frac{\dot{x}_s}{|\dot{x}_s|} = -\sin \left( \int_0^s \frac{a_u}{\a^2(t_u)} du \right) U_{\ii}  + \cos \left( \int_0^s \frac{a_u}{\a^2(t_u)} du \right) V_{\ii} + o(1)}.
\end{array}$$
Necessarily, we have then $|U_{\ii}| = |V_{\ii}|=1$ and $\langle U_{\ii} , V_{\ii} \rangle=0$, hence the result.
\end{proof}

\bibliographystyle{alpha}

\end{document}